\documentclass[12pt,a4paper,twoside,reqno]{amsart}

\usepackage{amssymb}
\usepackage[dvips]{graphicx}
\usepackage[ansinew]{inputenc}
\usepackage{a4wide}
\newtheorem{maintheorem}{Theorem}

\newcommand{\cmt}{\begin{maintheorem}}
\newcommand{\fmt}{\end{maintheorem}}

\newtheorem{maincorollary}[maintheorem]{Corollary}

\newcommand{\cmc}{\begin{maincorollary}}
\newcommand{\fmc}{\end{maincorollary}}

\newtheorem{T}{Theorem}[section]
\newcommand{\cte}{\begin{T}}
\newcommand{\fte}{\end{T}}

\newtheorem{Corollary}[T]{Corollary}
\newcommand{\cco}{\begin{Corollary}}
\newcommand{\fco}{\end{Corollary}}

\newtheorem{Proposition}[T]{Proposition}
\newcommand{\cpr}{\begin{Proposition}}
\newcommand{\fpr}{\end{Proposition}}

\newtheorem{Lemma}[T]{Lemma}
\newcommand{\cle}{\begin{Lemma}}
\newcommand{\fle}{\end{Lemma}}

\newcommand{\csle}{\begin{Sublema}}
\newcommand{\fsle}{\end{Sublemma}}

\theoremstyle{definition}
\newtheorem{Example}[T]{Example}
\newcommand{\cex}{\begin{Example}}
\newcommand{\fex}{\end{Example}}

\newtheorem{Remark}[T]{Remark}
\newcommand{\cre}{\begin{Remark}}
\newcommand{\fre}{\end{Remark}}

\newtheorem{Definition}[T]{Definition}
\newcommand{\cde}{\begin{Definition}}
\newcommand{\fde}{\end{Definition}}

\renewcommand{\th} {\theta}       
\newcommand\mycom[2]{\genfrac{}{}{0pt}{}{#1}{#2}}

\newcommand{\w} {\omega} 
\def \RR {{\mathbb R}}

\def \ZZ {{\mathbb Z}}
\def \NN {{\mathbb N}}

\def \TT {{\mathbb T}}

\newcommand{\Leb}{m}
\newcommand{\leb}{m}
\newcommand{\lip}{\operatorname{Lip}}
\newcommand{\dist}{\operatorname{dist}}
\newcommand{\supp}{\operatorname{supp}}
\newcommand{\cc}{{\mathcal C}}
\newcommand{\cp}{{\mathcal P}}

\title[Almost sure mixing rates for NUE]{Almost sure mixing rates for non-uniformly expanding maps}

\author{Xin Li}
\address{ICTP\\
Strada Costiera 11, 34151, Trieste, Italy}
\email{lxingmail@gmail.com}

\author{Helder Vilarinho}
\address{Universidade da Beira Interior\\
Rua Marqu\^es d'\'Avila e Bolama, 6200-001 Covilh\~a, Portugal}
\email{helder@ubi.pt} \urladdr{http://www.mat.ubi.pt/$\sim$helder}
\thanks{HV was partially supported by National Funds through FCT - ``Funda\c{c}\~{a}o para a Ci\^{e}ncia e a Tecnologia'', project PEst-OE/MAT/UI0212/2011. XL was supported by ICTP}



\begin{document}

\begin{abstract}
We consider random perturbations of non-uniformly expanding maps, possibly having a non-degenerate critical set. We prove that, if the Lebesgue measure of the set of points failing the non-uniform expansion or the slow recurrence to the critical set at a certain time, for almost all random orbits, decays in a (stretched) exponential fashion, then the decay of correlations along random orbits is stretched exponential, up to some waiting time. As applications, we obtain almost sure stretched exponential decay of random correlations for Viana maps, as for a class of non-uniformly expanding local diffeomorphisms and a quadratic family of interval maps.

\end{abstract}

\setcounter{tocdepth}{1}

\maketitle

\tableofcontents
\section{Introduction}
One of the interests of smooth ergodic theory is to study iterations of smooth maps $f$ on a Riemann manifold $M$ through measures $\mu$ preserved by $f$, which describes the asymptotic behaviours of typical orbits $\{f^k(x)\}_{k\in \NN}$, i.e., \begin{equation}\label{1}\frac{1}{n}\sum_{k=0}^{n-1}\delta_{f^k(x)}\xrightarrow{{n\to\infty}}\mu \end{equation}
in the weak$^*$ topology. The absolutely continuous (with respect to Lebesgue measure) invariant probability measures are of greatest relevance. If such measures are also ergodic, the statistical prediction given by \eqref{1} holds for a positive Lebesgue measure of initial states $x\in M$. The measures possessing such a rich property are called \emph{SRB} (Sinai-Ruelle-Bowen) \emph{measures}, and were introduced by Sinai, Ruelle and Bowen in \cite{Si72,R76,BR75,Bo75} for Anosov maps and Axiom A attractors.


If $\mu$ is mixing, we have \emph{the decay of correlations} $$\lim_{n\to\infty}\left(\int(\varphi\circ f^n)\psi d\mu-\int\varphi d\mu\int\psi d\mu\right)=0$$ for regular enough observables $\phi,\psi$. We are then interested in to describing how fast the convergence is. Among the available techniques we are particulary concerned to the use of induced schemes, popularized for this purposed by the works of L.-S. Young \cite{Y98,Y99}.

We are going to address random perturbation of dynamical systems, by replacing the original dynamics $f$ by a close map $f_t$, $t\in T$, chosen independently according to some probability law $\theta_{\epsilon}$ in $T$. In a nutshell, the study of correlations for random perturbations can follow two approaches. We could focus on the stationary measure $\mu_\epsilon$, which satisfies $$\int\varphi(f_t(x))\,d\mu_{\epsilon}(x)d\theta_\epsilon(t)=\int\varphi d\mu_\epsilon$$ for every continuous function $\varphi$. In this average ("annealed") setting, the correlation functions are expressed as
$$\int(\varphi\circ (f_{t_{n-1}}\circ\cdots\circ f_{t_0})(x))\psi(x)\, d\mu_\epsilon(x)d\prod_{i=0}^{n-1}d\theta_\epsilon(t_i)-\int\varphi d\mu_\epsilon\int\psi d\mu_\epsilon.$$
Essentially, we are averaging over all possible realizations, which due to the i.i.d. setting can be done by averaging at each time-step. This is the natural setting to formalize the correlations in terms of a Markov chain. Alternatively, we could look for an almost sure approach, by considering the product space $T^\ZZ$ and the usual skew product dynamics $S(\omega,x)=(\sigma(\omega),f_{\omega_0}(x))$ in $T^\ZZ\times M$. We focus now on the invariant probability measures for $S$ that disintegrates as $d\mu_{\omega}(x)d \theta_\epsilon^{\ZZ}(\omega)$. The correlation is expressed in the following fiberwise ("quenched") \emph{future} and  \emph{past random correlations}:
 \begin{align*}
 C_\w^+(\varphi,\psi,\mu, n)&=\left|\int(\varphi\circ f_\w^n)\psi\, d\mu_\w -\int\varphi\, d\mu_{\sigma^n(\w)}\int\psi\, d\mu_\w\right|,\\
 C_\w^-(\varphi,\psi,\mu, n)&=\left|\int(\varphi\circ f_{\sigma^{-n}(\w)}^n)\psi\, d\mu_{\sigma^{-n}(\w)}-\int\varphi\, d\mu_{\w}\int\psi\, d\mu_{\sigma^{-n}(\w)}\right|.
\end{align*}
There are several works dealing either with the annealed approach (e.g., \cite{BaV96,Ba97,BBD14}) or the almost sure (e.g. \cite{BaKS96,Ba97, Bu99,BBM02,AS07}).

This paper concerns to almost sure decay of correlations for random perturbations of non-uniformly expanding (NUE for short) maps. The strategy is to build induced Gibbs-Markov-Young structures with (stretched) exponential decay of recurrence times for random orbits. We use random Young towers and a coupling argument to estimate stretched exponential decay of correlations over the abstract random induced dynamic. This estimates gives rise to the almost sure decay of random correlations.

We give applications to some known families of NUE systems. We present new results respecting to Viana maps, which are a higher-dimensional maps with critical set given in \cite{V97}, and to an open class of local diffeomorphisms given in \cite{ABV00}. We also apply our results to the unimodal maps as considered in \cite{BBM02} to illustrate our strategy under the weaker hypotheses.

The strategy could be used for dynamical systems with other rates of decay (e.g. polynomial), but the hypotheses could be harder to achieve for known examples. Some other questions arise: Is the decay of correlations actually stretched exponential? Does a random central limit theorem hold? Do we have the parallel result in partially hyperbolic attractors admitting a non-uniformly expansion direction? Can we replace $\limsup$ in the definition of random non-uniformly expanding to the weaker assumption $\liminf$, as in the recent work \cite{ADLP}?

This paper is organized as follows. In \S2 we give our basic definitions of random perturbations for non-uniformly expanding and in \S3 we state the main results. We describe the main steps of our strategy in \S4 where we have the principal intermediate results, which are proved in \S6 (existence of random induced structures with stretched exponential decay of return times) and \S7 (decay of correlations for abstract random induced dynamics). The applications to Viana maps, local diffeomorphisms and unimodal interval maps is given at \S5.

\section{Random perturbations for non-uniformly expanding maps}
 Let $M$ be a compact Riemannian manifold endowed with a normalized
volume measure $m$ (Lebesgue measure), and $f\colon M\to M$ be a $C^2$ map. We assume that $f$ is a local diffeomorphism in the whole manifold  except, possibly, in a set $\cc\subset M$ containing the critical points of $f$ and $\partial M$. We say that the set $ \cc $ is {\em non-degenerate} if it has zero Lebesgue measure and there are constants $B>1$ and $\beta>0$ such that for every $x\in M\setminus\cc$
\begin{equation}\label{c1}
\displaystyle{\frac{1}{B}\dist(x,\cc)^{\beta}\leq \frac
{\|Df(x)v\|}{\|v\|}\leq B\dist(x,\cc)^{-\beta}}\,\,\forall v\in T_x
M,
\end{equation}
and, for every
$x,y\in M\setminus \cc$ with $\dist(x,y)<\dist(x,\cc)/2$, we have
\begin{equation}\label{c2}
\displaystyle{\left|\log\|Df(x)^{-1}\|-
\log\|Df(y)^{-1}\|\:\right|\leq
\frac{B}{\dist(x,\cc)^{\beta}}\dist(x,y)}
\end{equation}
\begin{equation}\label{c3}
\displaystyle{\left|\log|\det Df(x)|- \log|\det
Df(y)|\:\right|\leq \frac{B}{\dist(x,\cc)^{\beta}}\dist(x,y)}.
\end{equation}
Roughly speaking, condition \eqref{c1} says that
 $f$  {\em behaves like a power of the distance}
 to $ \cc $, meanwhile \eqref{c2} and \eqref{c3} say that the functions $  \log|\det Df| $ and $ \log \|Df^{-1}\|
$ are \emph{locally Lipschitz} in $ M \setminus \cc$,
with the Lipschitz constant depending on the distance to $\cc$. Given $\delta>0$ and $x\in M\setminus\cc$ we define the
{\em $\delta$-truncated distance\/} from $x$ to $\cc$ as $\dist_\delta(x,\cc)=\dist(x,\cc)$ if $\dist(x,\cc)<\delta$ and
$\dist_\delta(x,\cc)=1$ otherwise.


 The idea we adopt for random perturbations
is to
replace the original deterministic obits
by {\em random orbits} generated by an (independent and
identically distributed) random choice of
map at each iteration. We are interested in systems whose random perturbation exhibit a non-uniform expansive behavior along the orbits generated
 by the successive  composition of random elected maps.
To be more precise we consider a metric space $T$ and a continuous map
\begin{equation*}
 \begin{array}{rccl}
 \Phi
:& T &\longrightarrow&  C^2(M,M)\\
 & t &\longmapsto & \Phi(t)=f_t
 \end{array}
\end{equation*}
such that $f=f_{t^*}$ for some  $t^*\in T$, and a family $(\theta_\epsilon)_{\epsilon>0}$ of Borel probability measures in $T$. We will refer to such a pair $\chi_\epsilon=\{\Phi,(\th_\epsilon)_{\epsilon>0}\}$ as a {\em random
perturbation} of $f$.
We consider the product space $\Omega=T^\ZZ$ endowed with the Borel product probability measure $P=P_\epsilon=\theta_\epsilon^\ZZ$.
 For a {\em realization}
$\omega=(\ldots\omega_{-1},\omega_0,\omega_1,\ldots)\in \Omega$ and $n\geq 1$ we define
\begin{equation*}
f_{\omega}^n(x)=
 (f_{\omega_{n-1}}\circ\dots\circ f_{\omega_1}\circ f_{\omega_0})(x),
\end{equation*}
and set $f_\omega^0(x)=Id_M$. Given $x\in M$ and  $\w\in \Omega$ we call the sequence
$\big(f_\omega^n(x)\big)_{n\in\NN}$ a {\em random
orbit}\index{random orbit} of $x$. Note that
$\w^*=(\ldots,t^*,t^*,\ldots)$ gives rise to the unperturbed
deterministic orbits given by the
original dynamics
$f$.
We define the {\em two-sided skew-product} map
 \begin{equation*}
\begin{array}{rccc} S: & \Omega\times M &\longrightarrow &
\Omega\times M\\
 & (\omega, z) &\longmapsto & \big(\sigma(\omega),f_{\omega_0}(z)\big),
\end{array}
 \end{equation*}
 where $\sigma\colon \Omega \to \Omega$ is the left shift map.
  A Borel probability measure $\mu^*$ in $\Omega\times M$ invariant by  $S$ (in the usual deterministic sense) is characterized by an essentiality unique
disintegration $d\mu^*(\omega,x)=d\mu_{\omega}(x)
d P(\omega)$ given by a family $\{\mu_\w\}_{\w}$ of
{\it sample measures} on $M$ satisfying the quasi-invariance property ${f_{\omega}}_*\mu_\omega=\mu_{\sigma(\omega)}$, and such that for each Borel set $A\subset\Omega\times M$ we have $\mu^*(A)=\int\mu_\w(A_\w)\,d P(\w),$
where  $A_\w=\{x\in M: (\w,x)\in
A\}$.
For a complete introduction on random dynamical systems we refer for \cite{Arn98}.


Consider a random perturbation $\chi_\epsilon$ of a map $f$ such that
$$\supp(\theta_\epsilon)\rightarrow \{t^*\},\quad\text{as}\quad
\epsilon\to 0.$$
Due to the presence of the critical set, we assume that all the maps $f_t$
have the same critical set $\cc$:
 \begin{equation}\label{e.perturbation}
 Df_{t}(x)=Df(x), \quad\mbox{for every $x\in M\setminus\cc$ and $t\in T$}.
 \end{equation}
We could implement this setting, for instance, in parallelizable manifolds (with an
additive group structure, e.g. tori $\mathbb T^d$ (or cylinders $\TT^{d-k}\times
\RR^k$), by considering  $T=\{t\in\RR^d: \|t\|\leq\epsilon_0\}$
for some $\epsilon_0>0$, and
taking $f_t=f+t$, that is, adding at each step a random noise to the
unperturbed
dynamics.

\cde
We say that $f$ is
\emph{non-uniformly expanding on random orbits}  if the following conditions hold, at
least for small $\epsilon>0$:
\begin{enumerate}
  \item  there is ${\alpha}>0$ such that for
$P\times m$ almost every $(\omega,x)\in \Omega\times M$
 \begin{equation}\label{NUEOA}
    \limsup_{n\to{+\infty}}\frac{1}{n}\sum_{j=0}^{n-1}
    \log \|Df_{\sigma^j(\w)}({f_\w^{j}(x))}^{-1}\|<-{\alpha}.
\end{equation}
 \item  given any small ${\gamma} >0$ there is $\delta
>0$ such that for $P\times m$ almost every $(\omega,x)\in \Omega\times
M$
\begin{equation}
 \label{RecLentOA}
\limsup_{n\rightarrow +\infty}\frac{1}{n}
\sum_{j=0}^{n-1}-\log\dist_\delta(f^j_{\omega}(x), \mathcal C)< {\gamma}.
 \end{equation}
\end{enumerate}
 \fde
 When $\cc=\emptyset$
we simply disregard condition \eqref{RecLentOA} and assumption
\eqref{e.perturbation}.
We say that the original map $f$ is a \emph{non-uniformly expanding map} if \eqref{NUEOA} and \eqref{RecLentOA} holds for $\w^*$ and $\leb$ almost every $x$.
We will refer to the second condition by saying that the
random orbits of points have {\em slow recurrence}\index{slow
recurrence} to  $\cc$. Condition \eqref{NUEOA} implies
that for $P$ almost every (a.e.) $\omega\in \Omega$, the
\emph{expansion time} function
\begin{equation*}
\mathcal E_\omega(x) = \min\left\{N\ge1\colon
\frac1n\sum_{j=0}^{n-1} \log \|Df_{\sigma^j(\w)}({f_\omega^{j}(x))^{-1}}\| \leq
-{\alpha}, \text{ for all $n\geq N$}\right\}
\end{equation*}
is defined and finite Lebesgue almost everywhere in $M$.
We notice that condition \eqref{RecLentOA} is not
needed in all its strength, and we just need to ensure that it holds for suitable $\delta, \gamma$ so that the proof of Proposition \ref{l:hyperbolic2} works (see Remark 4.5 in \cite{AV10}). Hence, we may consider ${\gamma}>0$ and
$\delta>0$ such that
for $ P$ a.e. $\omega\in \Omega$ we can  define
the \emph{recurrence time} function Lebesgue almost everywhere in
$ M $:
\begin{equation*}
\mathcal R_\omega(x) = \min\left\{N\ge 1: \frac1n\sum_{j=0}^{n-1}
-\log \dist_\delta(f_\omega^j(x),\cc) \leq {\gamma} , \quad
\text{for all }n\geq N\right\}.
\end{equation*}
We
introduce the {\em tail set (at time
$n$)}\index{tail set}
\begin{equation}\label{tailset}
    \Gamma_\omega^n=\big\{x\in M: \mathcal E_\omega(x) > n \ \text{ or } \ \mathcal R_\omega(x) > n \big\}.
\end{equation}
This is the set of points in $ M $ whose random orbit at time $ n
$ has not yet achieved the uniform exponential growth of
derivative or the slow recurrence given by
conditions~\eqref{NUEOA} and~\eqref{RecLentOA}. If the critical
set is empty, we simply ignore the recurrence time function in the definition of
$\Gamma_\omega^n$.

\section{Main results}
 We assume that $f$ is a topologically transitive non-uniformly expanding
map and non-uniformly expanding on random orbits. We start with the case where we have a uniform (stretched) exponential decay of tail sets.

\cmt\label{thA} Assume there exist
$C, \gamma>0$ and $0<\upsilon\leq 1$ such that $\Leb(\Gamma_\omega^n)<Ce^{-\gamma n^\upsilon}$
for $P$ a.e. $\omega\in\Omega$. Then, if $\epsilon>0$ is small, for some integer $q\geq1$ we have:
\begin{enumerate}
\item[(i)] for $P$ a.e. $\w$ there is an absolutely continuous probability measure $\mu_\w=h_\w d\leb$ satisfying $(f^q_\w)_*\mu_\w=\mu_{\sigma^q(\w)}$;
\item[(ii)]   there exist $ C_i, \gamma_i>0$, $i=1,2$, and for $P$ a.e. $\w$ a positive integer $n_0(\w)$, such that for each Lipschitz function $\psi:M\to\RR$ and every bounded function $\varphi:M\to\RR$
  we have
$$     C_\w^{\pm}(\varphi,\psi,\mu,qn)\leq C_1\sup|\varphi|\lip(\psi)e^{-\gamma_1 n^{\upsilon/2}},\,  \forall\,n\geq n_0(\w)
$$
and
$$       P(\{n_0(\w)>n\})\leq C_2e^{-\gamma_2 n^{\upsilon/2}},\, \forall\,n\geq 1.$$
 \end{enumerate}\fmt
 We can interpret $n_0(\w)$ as the \emph{waiting time} we have to consider before we see the stretched exponential behavior on the estimates of the decay of random correlations.
 In many cases, the estimates on the tail can be hard to achieve, in particular its uniformity over distinct realizations. However, we can state the following.

 \cmt\label{thB} Assume that there exist
$C, \gamma>0$,  $0<\upsilon\leq 1$ and for $P$ a.e. $\w$ a positive integer $g_0(\w)$
 such that
\begin{equation*}
\left\{
\begin{array}{ll}
        \leb(\Gamma_\w^n)\leq Ce^{-\gamma n^\upsilon}, & \forall\,n\geq g_0(\w) \\
         P(\{g_0(\w)>n\})\leq  Ce^{-\gamma n^\upsilon},& \forall\,n\geq 1.
      \end{array}\right.
\end{equation*}
Then, if $\epsilon>0$ is small, for some integer $q\geq1$ we have:
\begin{enumerate}
\item[(i)] for $P$ a.e. $\w$ there is an absolutely continuous probability measure $\mu_\w=h_\w d\leb$ satisfying $(f^q_\w)_*\mu_\w=\mu_{\sigma^q(\w)}$;
\item[(ii)]  there exist $C_i, \gamma_i>0$, $i=1,2$, and for $P$ a.e. $\w$ a positive integer $n_0(\w)$ such that for each Lipschitz function $\psi:M\to\RR$ and every bounded function $\varphi:M\to\RR$
  we have
$$
C_\w^{\pm}(\varphi,\psi,\mu,qn)\leq C_1\sup|\varphi|\lip(\psi)e^{-\gamma_1 n^{\upsilon/4}}, \, \forall\,n\geq n_0(\w), $$
and $$  P(\{n_0(\w)>n\})\leq  C_2e^{-\gamma_2 n^{\upsilon/4}},\, \forall\,n\geq 1.$$
 \end{enumerate}\fmt

\begin{maincorollary}\label{CorolC}  Set $\Gamma^n=\{(\w,x)\in\Omega\times M:x\in\Gamma_\w^n\}$. Assume that there exist
$C, \gamma>0$ and  $0<\upsilon\leq 1$ such that
$$(P\times\leb) (\Gamma^n)<Ce^{-\gamma n^{\upsilon}}.$$
Then the same conclusions of Theorem \ref{thB} hold.
 \end{maincorollary}

\section{The strategy: an overview}

For both Theorems \ref{thA} and \ref{thB}, the proof consists in two main steps. First, we prove that the hypothesis on the tail sets imply the existence of induced structures
 with nice decay for the return times $R$. Moreover, as in the deterministic case, we will need a condition of type $\text{gcd} \{R\}=1$ in order to ensure some mixing properties in the induced  dynamics. In view of this, in this greater generality we need eventually to look for some power of the random system. If we are able to construct random induced structures with $\text{gcd} \{\textit{R}\}=1$ then the main results hold with $q=1$; see Remark \ref{r gcd}. We also notice that the hypothesis of transitivity are used for the existence of this suitable induced structures.
 In a second moment we give random versions of the already classic procedures introduced by Young \cite{Y98, Y99}, in order to derive the decay of correlations in the  induced dynamics from the decay of return times, that is later carried to the decay of correlations along random orbits. Corollary \ref{CorolC} is an immediate consequence of Theorem \ref{thB} and Lemma \ref{lemmatoC}.

 We notice that even in the case of uniform exponential decay of the tail set we do not achieve an exponential estimate for the decay of random correlations. The main reason for the damage in the estimates is related to the construction of induced measures, whose density is not bounded from below. Besides in some cases it is not known the optimal result on decay of correlations (even in the deterministic case), if we think for instance in the uniform expanding case it becomes clear that the strategy could compromise the search for optimal estimates.

\subsection{Random induced schemes}

We start by setting the induced Gibbs-Markov-Young (GMY for short) structures for the random orbits.
\cde
We say that \emph{$\w\in \Omega$ \textit{induces a GMY} map $F_\w$ in a
ball $\Delta\subset M$} if: \begin{enumerate}
\item there is a countable
partition $\cp_{\omega}$ of open sets of a full
$\leb$ measure subset $\mathcal{D}_\w$ of
$\Delta$;
\item there is a {\emph{return time function}}
$R_\omega:\mathcal{D}_\w\to\NN$, constant in each
$U_{\omega}\in\cp_{\omega}$;
\item the map $ F_{\omega}(x):=f_\w^{R_\w(x)}(x):\Delta\to\Delta$ verifies:
\begin{enumerate}
\item${ F}_{\omega}\vert_{U_{\omega}}$ is a $C^2$ diffeomorphism onto $\Delta$;
\item there exists $ 0<\kappa_\w<1$ such that for $x$
    in the interior of $U_{\omega}$ $$\|D
 F_{\omega}(x)^{-1}\| <\kappa_\w;$$
\item
    there is some constant $K_\w>0$ such that for every
    $U_{\omega}$ and $x,y\in U_{\omega}$
    \[
    \log\left|\frac{\det D F_{\omega}(x)}{\det D F_{\omega}(y)}\right| \leq
K_\w
    \dist( F_{\omega}(x), F_{\omega}(y)).
    \]
 \end{enumerate}
  \end{enumerate}
\fde

It is known that a deterministic transitive non-uniformly expanding map induces a GMY map $F$ (consider $\w=\w^*$ in the definition above) in some ball.
 The next theorem ensures that almost all realizations
induce a
 GMY map with some uniformity on the constants, and relates the decay of the return times with the the decay of the tail set.

\cte\label{random Markov structure}
Let $f:M\to M$ be a transitive non-uniformly expanding map and non-uniformly expanding on random orbits. There is some ball $\Delta\subset M$ such that if $\epsilon>0$ is small enough
then $P$ a.e. $\omega$ induces a GMY map $F_{\omega}$ in $\Delta$, and
 \begin{enumerate}
 \item[(i)] if there exist $C,\gamma>0$, $0<\upsilon\leq1$ such that
$\Leb(\Gamma_\omega^n)<Ce^{-\gamma n^\upsilon}$ for $P$ a.e. $\omega$,
  then there exist $C_1,\gamma_1>0$ such that $\leb(\{R_\w>n\})\leq C_1e^{-\gamma_1n^\upsilon}$;
%
\item[(ii)] if there exist
$C_i, \gamma_
i>0$, $i=1,2$,  $0<\upsilon\leq 1$ and for $P$ a.e. $\w$ a positive integer $g_0(\w)$
 such that
\begin{equation*}
\left\{
\begin{array}{ll}
        \leb(\Gamma_\w^n)\leq C_1e^{-\gamma_1 n^\upsilon}, & \forall\,n\geq g_0(\w) \\
         P(\{g_0(\w)>n\})\leq C_2e^{-\gamma_2 n^\upsilon},& \forall\,n\geq 1,
      \end{array}\right.
\end{equation*}
 then there exist $C_3, \gamma_3>0$,  $0<\upsilon\leq 1$
  such that
\begin{equation*}
\left\{
\begin{array}{ll}
       \leb(\{R_\w>n\})\leq C_3e^{-\gamma_3 n^\upsilon}, & \forall\,n\geq g_0(\w) \\
         P(\{g_0(\w)>n\})\leq C_2e^{-\gamma_2 n^\upsilon},& \forall\,n\geq 1,
      \end{array}\right.
\end{equation*}
\end{enumerate}

\fte
The proof is given in \S\ref{GMY}, where we can also deduce an induced GMY for $\w^*$ and the following uniformity conditions:
\begin{itemize}
\item[(U$_1$)] Given ${\xi}>0$ and an integer $\hat N>1$, if $\epsilon>0$ is sufficiently small then for $P$ a.e. $\omega$ we have
$$
m\left(\{R_{\omega}=j\}\triangle\{R_{\omega^*}
=j\}\right)\leq
{\xi} $$
 for $j=1,2,\ldots,\hat N$, where $\triangle$ stands for the symmetric difference of two sets.

\item[(U$_2$)] If $\epsilon>0$ is sufficiently small, the constants $K_\w$ and
$\kappa_\w$
for the induced GMY maps can be chosen uniformly over $\w$. We will
refer to them as $K>0$ and $\kappa >0$, respectively.
 \end{itemize}

\subsection{Decay of correlations for random induced schemes}

We start by following \cite{BBM02} and lift the GMY structures to random Young towers over copies of $\Delta$. Letting $\ZZ_+=\{0,1,2,\ldots\}$, set
$$\Delta_\w=\left\{ (x,\ell)\in\Delta\times\ZZ_+ : x\in\mathcal D_{\sigma^{-\ell}(\w)}, 0\leq \ell \leq R_{\sigma^{-j}(\w)}(x)-1 \right\}.$$
We set the $n$th level of the \emph{tower} $\Delta_\w$ as
$\Delta_{\w,\ell} = \Delta_\w \cap \{\ell=n\}$.
The basis of all towers are copies of $\Delta$, and $\ell$th  level $\Delta_{\w,\ell}$ is a copy of $\{x\in\Delta : R_{\sigma^{-\ell}(\w)}>\ell\}$. Moreover, for each basis $\Delta_{\w,0}$ we consider the corresponding partition $\cp_\w$ that can be extended to a partition of the full towers $\Delta_\w$. Moreover we denote by $\mathcal B_\w$ the Borel $\sigma$-algebra of $\Delta_\w$ and  by $\mathcal B$
 the corresponding family of $\sigma$-algebras $\mathcal B_\w$.
We define the  dynamics $F_\w:\Delta_\w\to\Delta_{\sigma(\w)}$  by
$$F_\w(x,\ell)=\left\{\begin{array}{lcl}(x,\ell+1), &\textrm{  if }&
\ell+1<R_{\sigma^{-\ell}(\w)}(x)\\(f_{\sigma^{-\ell}(\w)}^{R_{\sigma^{-\ell}(\w)}}(x),0), &\textrm{  if }& \ell+1=R_{\sigma^{-\ell}(\w)}(x).\end{array}\right.$$
This dynamics carries $(x,\ell)$ in the $\ell$th level of $\Delta_\w$ into $(x,\ell+1)\in\Delta_{\sigma(\w)}$, unless $R_{\sigma^{-\ell}(\w)}(x)=\ell+1$, in which case it falls down into the 0th level of $\Delta_{\sigma(\w)}$ by the return map.
We denote by $\leb_0$ de normalized Lebesgue measure on $\Delta$ and, without risk of confusion, we denote by $\leb$ the lift of this measure on $\Delta_\w$, also dropping the reference to $\w$ in the notation.  We set $\hat\Delta$ for the family $\{\Delta_\w\}_{\w}$ and $\cp$ for the corresponding partition introduced before.
We define, for almost every $\w$, the \emph{separation time} $s_\w:\Delta_\w\times\Delta_\w\to\ZZ_+\cup\{\infty\}$ given by
$$s_\w(x,y)=\min\{n\geq 0: F_\w^n(x)\text{ and } F_\w^n(y) \text{ lie in distinct elements of } \cp\}.$$
We introduce a {Lipschitz}-{type} space of observables $\varphi=\{\varphi_\w\}_\w$ on $\hat\Delta$,
\begin{equation*}
\mathcal{F}_\beta = \left\{\varphi\colon\hat\Delta\to\mathbb{R}\,: \exists\,
C_\varphi>0\,\text{s.t.}\, \left|{\varphi_\w(x)}-{\varphi_\w(y)}\right|\leq C_\varphi \beta^{s_\w(x,y)},\,\,
~~\forall
x,y \in \Delta_\w \right\},
\end{equation*}
a space of densities $\varphi=\{\varphi_\w\}_\w$,
\begin{align*}
\mathcal{F}_\beta^{+} = & \,\big\{ \varphi \in \mathcal{F}_{\beta}
\,: \exists\, \hat C_\varphi>0\mbox{ s.t. on each $U_\w\in\cp_\w $, either $\varphi_\w\vert_{U_\w}\equiv 0$, or}
\\  &
 \phantom{ggsdggggassdfasdfrt}\varphi_\w\vert_{U_\w}>0\textrm{ and }
\left|\log\frac{\varphi_\w(x)}{\varphi_\w(y)}\right|\leq \hat C_\varphi \beta^{s(x,y)},\,\,
~~\forall
x,y \in U_\w \big\},
\end{align*}
and a space of random bounded functions $\varphi=\{\varphi_\w\}_\w$,
$$\mathcal L^\infty =\{\varphi:\hat\Delta\to\mathbb R \,:\, \exists\, \tilde C_\varphi>0\, \textrm{s.t.}\, \sup_{x\in\Delta_\w} |\varphi_\w|\leq \tilde C_\varphi\}.$$
Let us assume the conclusions of Theorem \ref{random Markov structure}, either in the case $(i)$ of uniform estimates on return times or case $(ii)$ of non-uniform rates.

\cte\label{exist.mu.induced} For $P$ a.e. $\omega$ there is an absolutely continuous probability measure $\nu_\w$ on $\Delta_\w$ such that
${(F_\w)}_*\nu_\w=\nu_{\sigma(\w)}$. Moreover, $\rho_{\omega}=d\nu_{\omega}/d\leb\in\mathcal{F}_\beta^+$ and there is a constant $K_1>0$ such that
$\rho_{\omega}\leq K_1$ for $P$ a.e. $\w$.\fte

For the proof see \cite{BBM02, BaBeM10}. Roughly speaking, for each $\omega$ and $n\ge 0$, we consider
 the push-forward $\nu^n_\omega$ of
$\leb_0|\Delta_{\sigma^{-n}(\omega),0}$ by
$F^n_{\sigma^{-n}(\omega)}$.
This push-forward is a probability measure on the tower
$\Delta_\omega$, absolutely continuous with
respect to $m$. One can see (following, for instance, estimates (3.9) in \cite{BBM02}) that the densities $\varphi^{n}_{\omega}$
of the
$\nu^n_\omega$ belong to $\mathcal F^+_\beta$, with constants $C_{\varphi_\omega^n}$ depending only on $K$ and $\kappa$ as in our condition (U$_2$). The hypotheses on the decay of return times imply that
 for $P$ a.e.
 $\omega$ there is
a subsequence $n_\ell \to \infty$
so that $\frac{1}{n_\ell}\sum_{k=0}^{n_\ell-1}\nu^{k}_{\omega}$
converges in the weak$^*$ topology to
a probability measure on $\Delta_\omega$, absolutely continuous with respect
to $m$.
 By a diagonalization argument, for $P$ a.e. $\w$ we can find a sequence $n_j$ such that, for each integer $N$, $\frac{1}{n_j}\sum_{k=N}^{n_j-1} \nu^{k-N}_{\sigma^{-N}(\omega)}
 $
converges to
a probability measure $\nu_{\sigma^{-N}(\omega)}$
on the tower $\Delta_{\sigma^{-N}(\omega)}$ with the desired properties.\\

Let us now define the correlations on $\hat\Delta$. Given a family of measures $\nu=\{\nu_\w\}$ in $\hat\Delta$ and $\varphi, \psi:\hat\Delta\to\mathbb R$, we set the \emph{future correlation} as
\begin{equation*}
\bar C^+_{\w}(\varphi,\psi,\nu,n)=\left|\int_{\Delta_\w}(\varphi_{\sigma^n(\w)}\circ F_\w^n)\psi_\w\,d\nu_{\w} -\int_{\Delta_{\sigma^n(\w)}}\varphi_{\sigma^n(\w)}\,d\nu_{\sigma^n(\w)}\int_{\Delta_\w}\psi_\w\,d\nu_{\w} \right|
\end{equation*}
and, similarly, the \emph{past correlation}
\begin{equation*}
\bar C^-_{\w}(\varphi,\psi,\nu,n)=\left|\int_{\Delta_{\sigma^{-n}(\w)}}\!\!\!(\varphi_{\w}\circ F_{\sigma^{-n}(\w)}^n)\psi_{\sigma^{-n}(\w)}\,d\nu_{\sigma^{-n}(\w)} -\int_{\Delta_{\w}}\!\!\varphi_{\w}\,d\nu_{\w}\!\!\int_{\Delta_{\sigma^{-n}(\w)}}\!\!\!\!\psi_{\sigma^{-n}(\w)}\,d\nu_{\sigma^{-n}(\w)} \right|.
\end{equation*}
We relate now the estimates on return times with the induced decay of correlations.
\cte\label{dacayrate}
Let $\varphi=\{\varphi_\w\}\in\mathcal L^\infty$ and $\psi=\{\psi_\w\}\in\mathcal F_\beta$.
\begin{enumerate}
\item[(i)] If there exist $C_1,\gamma_1>0$ and $0<\upsilon\leq1$ such that
$\Leb(R_\w>n)\leq C_1e^{-\gamma_1 n^\upsilon}$
  then there exist $C_i, \gamma_i>0$, $i=1,2$, and for $P$ a.e. $\w$ a positive integer $n_0(\w)$, such that
\begin{equation*}
\left\{
\begin{array}{ll}
        \bar C^\pm_{\w}(\varphi,\psi,\nu,n)\leq C_2e^{-\gamma_2 n^{\upsilon/2}}, & \forall\,n\geq n_0(\w) \\
         P(\{n_0(\w)>n\})\leq C_3e^{-\gamma_3 n^{\upsilon/2}},& \forall\,n\geq 1.
      \end{array}\right.
\end{equation*}
\item[(ii)] If there exist $C_i, \gamma_i>0$, $i=1,2$, and $0<\upsilon\leq1$, and for $P$ a.e. $\w$ a positive integer $g_0(\w)$, such that
\begin{equation*}
\left\{
\begin{array}{ll}
        \leb(\{R_\w>n\})\leq C_1e^{-\gamma_1 n^\upsilon}, & \forall\,n\geq g_0(\w) \\
         P(\{g_0(\w)>n\})\leq C_2e^{-\gamma_2 n^\upsilon},& \forall\,n\geq 1,
      \end{array}\right.
\end{equation*}
then there exist $C_j, \gamma_j>0$, $j=3,4$, and for $P$ a.e. $\w$ a positive integer $n_0(\w)$ such that
\begin{equation*}
\left\{
\begin{array}{ll}
       \bar
 C^\pm_{\w}(\varphi,\psi,\nu,n)\leq C_3e^{-\gamma_3 n^{\upsilon/4}}, & \forall\,n\geq n_0(\w) \\
         P(\{n_0(\w)>n\})\leq C_4e^{-\gamma_4 n^{\upsilon/4}},& \forall\,n\geq 1.
      \end{array}\right.
\end{equation*}
 \end{enumerate}
\fte
Theorem \ref{dacayrate} is proved in \S\ref{Decay of correlations for random GMY}.

\subsection{Going back to the initial dynamics}
Finally, we transpose the estimates for correlations in the induced dynamics back to the original perturbed system, and finish the proof of Theorems \ref{thA} and \ref{thB}.
 Consider the projection $\pi_\w\colon\Delta_\w\to M$ given by $\pi_\w(x,\ell)=f_{\sigma^{-\ell}(\w)}^\ell(x)$, which satisfies $f_\w\circ\pi_\w=\pi_{\sigma(\w)}\circ F_\w$.   Let $\nu=\{\nu_{\omega}\}$ be a family of absolutely continuous probability  measures given by Theorem \ref{exist.mu.induced}. We could define the family  $\mu=\{\mu_\w\}$ of absolutely continuous sample probability measures on $M$ by $\mu_\w=({\pi_\w})_*\nu_\w$,
Note that if $\varphi:M\to\mathbb R$ is a Lipchitz function then the lifted functions $\tilde\varphi_\w:\Delta_\w\to\mathbb R$ given by $\varphi\circ\pi_\w$ belong to $\mathcal F_\beta$, with $\beta$ depending on $\kappa$  and $C_{\tilde\varphi}$ depending linearly on $\textrm{Lip}(\varphi)$.
  On the other hand, if $\varphi$ is bounded then $\tilde\varphi\in\mathcal L^\infty$ and $\sup_\Delta|\tilde\varphi|\leq \sup|\varphi|$, so that we may consider $\tilde C_{\tilde\varphi}\leq\|\varphi||_\infty$. It is straightforward that the estimates from the induced schemes are similarly translated to the original dynamics since $
C^{\pm}_{\w}(\varphi,\psi,\mu,n)=\bar C^{\pm}_{\w}(\tilde\varphi,\tilde\psi,\nu,n)$.

The proof of Theorems \ref{thA} and \ref{thB} follows now from Theorem \ref{random Markov structure} and Theorem \ref{dacayrate}.
Finally, let us make some considerations on the unicity of the family of sample probability measures. From \cite{AV10} one knows that the topological transitivity of $f$ leads to the unicity
 of the absolutely continuous stationary probability measure. Moreover, the $S$-invariant probability measures characterised by an essentiality unique family of sample probability measures, and they are in a one-to-one correspondence with the stationary probability measures, so that different families of sample probability measures would correspond to different stationary measures.

\section{Applications}

\subsection{Local diffeomorphisms}

We recall a
robust ($C^1$ open) classes of non-uniformly expanding local diffeomorphisms (with no
critical set) introduced in  \cite{ABV00}. The existence and unicity of SRB pro\-ba\-bi\-li\-ty measures for this maps was proved in \cite{ABV00,A03}. This classes of maps  and can be obtained, e.g.
through
deformation of a uniformly expanding map by isotopy inside some
small region. In general, these maps are not uniformly expanding:
deformation can be made in such way that the new map has periodic
saddles. Random perturbations for this maps were considered
in \cite{AAr03,AV10}. We obtain exponential estimates for the decay of correlations along the random orbits. For estimates on  decay of correlations for random perturbations of uniformly expanding maps see \cite{BaKS96}.\\

 Let $M$ be the  $d$-dimensional torus $\mathbb{T}^d$, for some $d\ge 2$, and $m$
  the normalized Riemannian volume form. Let  $f_0\colon M\to M$ be a uniformly expanding map and $V\subset M$ be a small
neighborhood of a fixed point $p$ of $f_0$ so that the restriction of $f_0$ to $V$ is
injective. Consider a $C^1$-neighborhood $ {\mathcal U}$ of $f_0$ sufficiently small so that
any map $f\in {\mathcal U}$ satisfies:
\begin{enumerate}
\item[(\emph{i})] $f$ is {\em expanding outside $V$}: there
exists $\lambda_0<1$ such that
 $$\|Df(x)^{-1}\|< \lambda_0\quad\text{for
every $x\in M\setminus V$};$$
 \item[(\emph{ii})] $f$ is {\em volume expanding
everywhere}\index{expanding!volume}: there exists $\lambda_1>1$
such that
 $$|\det Df(x)| > \lambda_1\quad\text{ for every $x\in M$;}$$

 \item[(\emph{iii})]
 $f$ is {\em not too contracting on~$V$}: there is some
 small $\gamma>0$ such that
$$\|Df(x)^{-1}\|< 1+\gamma\quad\text{ for every $x\in
V$};$$
 \end{enumerate}
 and, moreover, the constants $\lambda_0, \lambda_1$ and $\gamma$ are the same for all
$f\in{\mathcal U}$.

It was shown in \cite{AV10} how to perform the construction a bit more carefully in order to have topologically mixing maps, and thus transitive, by considering a map $\bar f\colon M\to M$ (in the boundary of the set of uniformly expanding maps) which satisfies (\emph{i}), (\emph{ii}) and (\emph{iii}) as the cartesian product of  one-dimensional maps $\varphi_1\times\cdots\times\varphi_d$, with $\varphi_1,\dots,\varphi_{d-1}$ uniformly expanding in $S^1$, and $\varphi_d$ the \emph{intermittent} map in $S^1$: it can be written as
 $$\varphi_d(x)=x+x^{1+\alpha}, \quad \text{for some $0<\alpha<1$,}$$
in a neighborhood of 0  and $\varphi_d'(x)>1$ for every $x\in S^1\setminus\{0\}.$
 If $f$ is in a sufficiently small $C^1$-neighborhood
 $\bar{\mathcal U}$ of $\bar f$, it satisfies (\emph{i}), (\emph{ii}) and (\emph{iii}) for convenient choice of constants
 $\lambda_0$, $\lambda_1$, $\gamma$ and a neighborhood $V$ of the fixed point $p=0\in \TT^d$, and is topologically mixing.

For $f\in{\mathcal U}$ we introduce random perturbations
$\{\Phi,(\theta_{\epsilon})_{\epsilon>0}\}$. In particular,
we consider a
continuous map
  \[
  \begin{array}{rccl}
  \Phi:& T &\longrightarrow&  {\mathcal U}\\
  & t &\longmapsto & f_t
  \end{array}
  \]
 where $T$
is a metric space and $f\equiv f_{t^*}$ for some
$t^*\in T$. Consider
a family
$(\theta_\epsilon)_{\epsilon>0}$ of
probability measures on $T$ such that their supports are non-empty and satisfies $
 \supp(\theta_\epsilon)\rightarrow \{t^*\}$, {when} $\epsilon\to 0$.
 According to \cite{AAr03}, we can choose appropriately the constants
$\lambda_0$, $\lambda_1$ and $\gamma$ so
that every map $f\in{\mathcal U}$ is non-uniformly
expanding on {\em all} random orbits with uniform exponential decay of the Lebesgue measure of the tail sets $\Gamma_\w^n$ given by \eqref{tailset}, ignoring naturally the recurrence time function:
 \begin{Proposition}
 Consider $f\in{\mathcal U}$ and $\{\Phi,(\theta_{\epsilon})_{\epsilon>0}\}$
as before. There exists
$\alpha>0$ such that for every $\w\in \supp(\theta_\epsilon^\NN)$ and $\leb$ a.e. $x\in M$
\begin{equation*}
\limsup_{n\to+\infty}\frac1n\sum_{j=0}^{n-1}
\log\| Df_{\sigma^j(\w)}(f^j_{\w}(x))^{-1}\| \le
-\alpha.\end{equation*}
Moreover, there is $0<\tau<1$ such that
 $m(\Gamma_\w^n)\leq \tau^n$,
for $n\ge 1$ and $\w\in \supp(\theta_\epsilon^\NN)$.
 \end{Proposition}
As application of Theorem \ref{thA} we have the following.
\cte  Let $f\in\bar{\mathcal U}$. Then, for some integer $q\geq1$ and all sufficiently small $\epsilon>0$:
\begin{enumerate}
\item[(i)] for $P$ a.e. $\w$ there is an absolutely continuous probability measure $\mu_\w=h_\w d\leb$ satisfying $(f^q_\w)_*\mu_\w=\mu_{\sigma^q(\w)}$;
\item[(ii)]   there exist $ C_i, \gamma_i>0$, $i=1,2$, and for $P$ a.e. $\w$ a positive integer $n_0(\w)$, such that for each Lipschitz function $\psi:M\to\RR$ and every bounded function $\varphi:M\to\RR$
  we have
$$     C_\w^{\pm}(\varphi,\psi,\mu,qn)\leq C_1\sup|\varphi|\lip(\psi)e^{-\gamma_1 \sqrt n},\,  \forall\,n\geq n_0(\w)
$$
and
$$       P(\{n_0(\w)>n\})\leq C_2e^{-\gamma_2 \sqrt n},\, \forall\,n\geq 1.$$

 \end{enumerate}

\fte

\subsection{Viana maps}

We consider now an important open class of non-uniformly expanding maps with
critical sets in higher
dimensions introduced in \cite{V97}.
Without loss of generality we discuss the
two-dimensional
case and we refer \cite{V97,AV02,AAr03} for details.\\

Let $p_0\in(1,2)$ be such that the
critical point
$x=0$ is pre-periodic for the quadratic map $Q(x)=p_0-x^2$. Let
$S^1=\RR/\ZZ$ and $b:S^1\rightarrow \RR$ be a Morse function, for
instance, $b(s)=\sin(2\pi s)$. For fixed small $\alpha>0$,
consider the map
 $$ \begin{array}{rccc} \hat f: & S^1\times\RR
&\longrightarrow & S^1\times \RR\\
 & (s, x) &\longmapsto & \big(\hat g(s),\hat q(s,x)\big)
\end{array}
 $$
 where $\hat g$ is the uniformly expanding
map of the circle defined by $\hat{g}(s)=ds$ (mod $\ZZ$) for some
$d\ge16$, and $\hat{q}(s,x)=a(s)-x^2$ with $a(s)=p_0+\alpha b(s)$. As it is shown in \cite{AAr03}, it
is no restriction to assume that $\cc=\{(s,x)\in S^1\times I\colon
x=0\}$ is the critical set of $\hat f$  and we do so.
If
$\alpha>0$ is small enough there is an
interval $I\subset (-2,2)$ for which $\hat f(S^1\times I)$ is
contained in the interior of $S^1\times I$. Any map $f$
sufficiently close to $\hat f$ in the $C^3$ topology has
$S^1\times I$ as a forward invariant region (in fact, here it suffices to be
$C^1$
close). We consider a small $C^3$ neighborhood $\mathcal V$ of $\hat f$ as
before and will refer to maps in $\mathcal V$ as {\em Viana maps}. Thus, any
{\em Viana map} $f\in\mathcal V$ has $S^1\times I$ as a forward invariant region,
and so an {\em attractor} inside it, which is precisely $$\Lambda=\bigcap_{n\geq
0}f^n(S^1\times I).$$
In \cite[Theorem C]{AV02} it was proved
a topological mixing property.
\cpr
For every $f\in\mathcal V$ and every open set $A\subset
S^1\times I$ there is some $n_A\in\NN$ for which $f^{n_A}(A)=\Lambda$.
\fpr

 We introduce the random perturbations $\{\Phi,(\theta_\epsilon)_\epsilon\}$ for this maps. We set
   $T\subset\mathcal V$
to be a $C^3$
neighborhood of $\hat{f}$ consisting in maps $f$ restricted to the
forward invariant region
$S^1\times I$ for which
$Df(x)=D\hat f(x)$ if $x\notin\mathcal C$, the map $\Phi$ to be the identity map at $T$ and
$(\theta_\epsilon)_{\epsilon}$ a family of Borel measures on $T$ such that their supports are non-empty and satisfy $
 \supp(\theta_\epsilon)\rightarrow \{f\}$, {when} $\epsilon\to 0$, for  $f\in T$. In \cite{AAr03} the authors realized that \emph{Viana maps} are non-uniformly expanding and non-uniformly expanding on random orbits, and that there exist $C, \gamma>0$ such that
$\leb(\Gamma_\w^n)<Ce^{-\gamma\sqrt n}$, for almost every $\w\in\supp(\theta_\epsilon^\NN)$. We may conclude the following from Theorem \ref{thA}.

 \cte
  Let $f\in\mathcal V$ be a Viana map. Then, for some integer $q\geq1$ and all sufficiently small $\epsilon>0$:
\begin{enumerate}
\item[(i)] for $P$ a.e. $\w$ there is an absolutely continuous probability measure $\mu_\w=h_\w d\leb$ satisfying $(f^q_\w)_*\mu_\w=\mu_{\sigma^q(\w)}$;
\item[(ii)]   there exist $ C_i, \gamma_i>0$, $i=1,2$, and for $P$ a.e. $\w$ a positive integer $n_0(\w)$, such that for each Lipschitz function $\psi:M\to\RR$ and every bounded function $\varphi:M\to\RR$
  we have
$$     C_\w^{\pm}(\varphi,\psi,\mu,qn)\leq C_1\sup|\varphi|\lip(\psi)e^{-\gamma_1 n^{1/4}},\,  \forall\,n\geq n_0(\w)
$$
and
$$       P(\{n_0(\w)>n\})\leq C_2e^{-\gamma_2 n^{1/4}},\, \forall\,n\geq 1.$$

 \end{enumerate}
  \fte

\subsection{Unimodal maps} We consider now random perturbations for a class of unimodal maps as in \cite{BBM02}. In that paper the authors construct directly the induced structures with non-uniform decay of return times, meanwhile we are going to check the hypothesis of Theorem \ref{thB}. The improvements in the stretched exponential rates for the decay of random correlations as compared with \cite{BBM02} are not relevant, and the main motivation is to illustrate our techniques in a case where we are not in conditions to obtain uniform estimates for the decay of the tail sets along random orbits. In \cite{BaV96} the authors consider a similar family of unimodal maps, for which obtain uniform exponential decay of correlations with respect to the stationary measure, in contrast to our almost sure results.\\

We start by recalling the setting and refer for \cite{BBM02} for details. Let $I = [L,R]$ be a compact interval containing $0$ in its interior and $f \colon I\to I$ be a $C^2$ unimodal
map ($f$ is increasing on $[L, 0]$ and decreasing on $[0,R]$) satisfying: $f''(0)\neq0$,  $\sup_I|f'| < 8$, and
\begin{itemize}
\item[(H1)] There are $0<\alpha<1$ and $1<\lambda\leq 4$ with $200\alpha<(\log \lambda)^2$ for which
\begin{enumerate}
\item [(i)] $|(f^n)'(f(0))| \geq \lambda^n$, for all $n\geq0$.
\item [(ii)] $|f^n(0)| \geq e^{-\alpha n}$, for all $n\geq0$.
\end{enumerate}
\item[(H2)] For each small enough $\delta>0$, there is $N =N(\delta) \geq 0$ for which
\begin{enumerate}
\item [(i)] If $x,\ldots, f^{N-1}(x)\notin(-\delta,\delta)$ then $|(f^N)'(x)| \geq \lambda^N$.
\item [(ii)] For each $n$, if $x,\ldots,f^{n-1}(x)\notin(-\delta,\delta)$ and $f^n(x)\in(-\delta, \delta)$, then $|(f^n)'(x)| \geq \lambda^n$.
\end{enumerate}

\item[(H3)] $f(I)$ is a subset of the interior of $I$.
\item[(H4)] $f$ is topologically mixing on $[f^2(0), f(0)]$.
\end{itemize}
Examples of unimodal maps satisfying this hypothesis are the quadratic maps $f_a(x)=a-x^2$ for a positive
Lebesgue measure set of (\emph{Benedicks-Carleson}, \cite{BeC85}) parameters $a$.

Fixing $\epsilon_0 > 0$ small enough to guarantee $f(x)\pm\epsilon_0\in I$ for all $x\in I$, we assume that we are
given a constant $D >0$ and for each $0 < \epsilon < \epsilon_0$ a probability measure $\theta_\epsilon$ on $T=T_\epsilon=[-\epsilon, \epsilon]$ such that for
any subinterval $J\subset T$,
\begin{equation}\label{q dens}
\theta_\epsilon(J) \leq \frac{D|J|}\epsilon
\end{equation}
Assumption \eqref{q dens} may be relaxed, but it cannot be completely suppressed since there are open intervals of parameters
corresponding to periodic attractors arbitrarily close to $a$. Assumption \eqref{q dens} holds for instance if $\theta_\epsilon$ has a
density with respect to Lebesgue which is bounded above by $D/\epsilon$. We stress that this does not imply that $0$
belongs to the support of $\theta_\epsilon$. We consider $\Omega=\Omega_\epsilon = T^\mathbb{Z}$, $P=P_\epsilon=\theta_\epsilon^\mathbb{Z}$ and for $t\in T$ we set $\Phi(t)=f_t(x)=f(x)+t$. We assume that $f$ is a transitive non-uniformly expanding and non-uniformly expanding (with slow recurrence to the critical set $\mathcal C=\{0\}$) and notice that it holds for $f_a$ with \emph{Benedicks-Carleson} parameter $a$; see \cite{F05}. We will see that we are in conditions to apply Theorem \ref{thB} to get the following result.

\cte\label{quadratic}
For some integer $q\geq1$, if $\epsilon>0$ is small:
\begin{enumerate}
\item[(i)] for $P$ a.e. $\w$ there is an absolutely continuous probability measure $\mu_\w=h_\w d\leb$ satisfying $(f^q_\w)_*\mu_\w=\mu_{\sigma^q(\w)}$;
\item[(ii)]  there exist $C_i, \gamma_i>0$, $i=1,2$, and for $P$ a.e. $\w$ a positive integer $n_0(\w)$ such that for each Lipschitz function $\psi:I\to\RR$ and every bounded function $\varphi:I\to\RR$
  we have
$$
C_\w^{\pm}(\varphi,\psi,\mu,qn)\leq C_1\sup|\varphi|\lip(\psi)e^{-\gamma_1 n^{1/4}}, \, \forall\,n\geq n_0(\w), $$
and $$  P(\{n_0(\w)>n\})\leq  C_2e^{-\gamma_2 n^{1/4}},\, \forall\,n\geq 1.$$
  \end{enumerate}
\fte

\subsubsection{Non-uniform expansion and slow recurrence} In \cite[\S7.1 and \S7.2]{BBM02} the authors followed some ideas from \cite{V97} and also \cite{A00} in order to have some estimates on the recurrence near the critical set for random orbits of points in the interval. We are going to discuss how to translate those estimates to our framework.
 Consider $\eta>0$ such that $$\frac{2\alpha}{\log \lambda}<\eta<\frac1{10}.$$
For $r\in\mathbb Z_+$ let $I_r=(\sqrt\epsilon e^{-r},\sqrt\epsilon e^{-(r-1)})$, and $I_r=-I_{|r|}$ for $r\leq-1$. For $k\geq 1$ we introduce the functions $r_k\colon \Omega\times I\to\mathbb Z_+$, by
setting $r_k(\w,x)= |r|$ if $f_\w^k(x)\in I_r$ and $r_k(\w,x)=0$ otherwise, and sets
$$G_k(\w,x) = G_k^\epsilon(\w,x) =\left\{
0 \leq j \leq k :  r_j(\w,x) \geq \max\left\{1,\left(\frac12-2\eta\right)\log(1/\epsilon)\right\}\right\}.$$
There are suitably small $c>0$ and large $C>1$ such that for sufficiently small $\epsilon>0$, large enough $n\gg C\log(1/\epsilon)$ and all $(\w,x)$ for which
\begin{equation}\label{hyp times}\sum_{ j\in G_m^\epsilon(\w,x)}r_j(\w,x)\leq cn,\end{equation}
we have $|(f_\w^n)'(x)|>e^{n/C}$. From \cite[Corollary 7.5]{BBM02} we have the following.

\cle\label{7.4 e 7.5}
There are $C(\epsilon)>1$, $\gamma(\epsilon) > 1/(C \log(1/\epsilon))$, and for each $n \geq 1$ there are sets $E_n\subset\Omega\times I$ with
$(P\times\leb)(E_n)\leq C(\epsilon)e^{-\gamma(\epsilon)n}$, such that if $(\w,x)\notin E_n$ then condition \eqref{hyp times} holds.
\fle
From Lemma \ref{7.4 e 7.5} we have $\sum_{n\geq1} (P\times\leb) (E_n)<\infty$, and by Borel-Cantelli's lemma
$$(P\times\leb)\left(\bigcap_{n\geq1}\bigcup_{k\geq n} E_k\right)=0.$$
This means that for $(P\times\leb)$ a.e. $(\w,x)$ condition \eqref{hyp times} holds for every $n$ large enough, implying that $|(f_\w^n)'(x)|>e^{n/C}$, and thus that $f$ is non-uniformly expanding on random orbits. Moreover, given $\zeta>0$ if we set $$E_n^\zeta=\left\{(\w,x):\sum_{0\neq j\in G_m^\epsilon(\w,x)}r_j(\w,x)\geq \zeta n,\right\},$$
then, for small $\zeta>0$, $(P\times\leb
)(E_n^\zeta)\leq C(\epsilon)e^{-\gamma(\epsilon)n}$. If we take $\delta=(1/2-2\eta)\log(1/\epsilon)$, then, for $(\w,x)\notin E_n^\zeta$ we have
$$\frac1n\sum_{j=0}^{n-1}-\log\dist_\delta(f_\w^j(x),\mathcal C)\leq\zeta .$$
Proceeding as before, we may conclude that $f$ has slow recurrence to the critical set on random orbits. The tail sets can be considered as $\Gamma_\w^n=\{x:(\w,x)\in E_n\cup E_n^\zeta\}\subset I$. The rate of decay of the Lebesgue measure of this sets is estimated to be exponential, but not uniform on $\w$: by Fubini's theorem there exists $C=C(\epsilon), \gamma=\gamma(\epsilon)>0$ such that
\begin{equation*}
(P\times\leb)(\Gamma^n)\leq Ce^{-\gamma n},
\end{equation*}
 where $\Gamma^n=\{(\w,x):x\in\Gamma_\w^n\}$. With the following lemma we are in conditions to apply Theorem \ref{thB} and conclude Theorem \ref{quadratic}.

\cle\label{lemmatoC}
If there exist $C, \gamma >0$ and $0<\upsilon\leq1$ such that for all $n\geq 1$
\begin{equation*}
(P\times\leb) (\Gamma^n)<Ce^{-\gamma n^{\upsilon}}
\end{equation*} then there exist $C_i, \gamma_i>0$, $i=1,2$, and for $P$ a.e. $\w$ a positive integer $g_0(\w)$ such that
\begin{equation*}
\left\{
\begin{array}{ll}
        \leb(\Gamma_\w^n)\leq C_1e^{-\gamma_1 n^{\upsilon}}, & \forall\,n\geq g_0(\w) \\
         P(\{g_0(\w)>n\})\leq C_2e^{-\gamma_2 n^{\upsilon}},& \forall\,n\geq 1.
      \end{array}\right.
\end{equation*}
\fle

\begin{proof}
By Fubini's theorem,
\begin{equation}\label{fubini}
P\left(\left\{\w:\leb(\Gamma_\w^n)>\sqrt{Ce^{-\gamma n^\upsilon}}\right\}\right)\leq \sqrt{Ce^{-\gamma n^\upsilon}},
\end{equation}
otherwise we are lead into a contradiction: $(P\times\leb)(\Gamma^n)=\int_\Omega\leb(\Gamma_\w^n)\,dP>{Ce^{-\gamma n^\upsilon}}$.
Set
\begin{equation*}
B_n=\left\{\w:\exists m\geq n\,\, \text{s.t.} \,\, \leb(\Gamma_\w^m)>\sqrt{Ce^{-\gamma m^\upsilon}}\right\}.
\end{equation*}
By \eqref{fubini}, we have
\begin{equation*}
P(B_n)\leq\sum_{k=n}^\infty P\left(\left\{\w:\leb(\Gamma_\w^k)>\sqrt{Ce^{-\gamma k^\upsilon}}\right\}\right)<\infty,
\end{equation*}
so that $\lim_{n\to\infty} P(B_n)=0$. For $\w$ in the $P$ full measure subset $\cup_n(\Omega\setminus B_n)$ of $\Omega$ we define
\begin{equation*}
g_0(\w)=\min\{ n\geq 0: \w\notin B_n\}.
\end{equation*}
\end{proof}

\section{Random Gibbs-Markov-Young structures}\label{GMY}

In this part we get the random GMY structure for the random NUE system. We prove Theorem \ref{random Markov structure} in \S\ref{s.tail}. We firstly derive uniform expansion and bounded distortion, then we simulate the GMY structure given in \cite{AL13} for partially hyperbolic attractors with non-uniformly expanding direction. Fix $B>1$ and
$\beta>0$ as in the definition of the critical set $\mathcal C$, and take a constant $b>0$ such that $2b <
\min\{1,\beta^{-1}\}$.
\cde
For $0<\lambda<1$, $\delta>0$, we call $n$ a
{\em $(\lambda,\delta)$-hyperbolic time} for  $(\omega, x)\in \Omega\times M$ if for all $1\le k \le n$
$$
\prod_{j=n-k+1}^{n}\|Df_{\sigma^j(\omega)}(f^j_{\omega}(x))^{-1}\| \le \lambda^k,
{\quad\text{and}\quad} \dist_{\delta}(f^{n-k}_{\omega}(x),\mathcal C)\ge\lambda^{bk}.
$$
\fde
In the case of $\cc=\emptyset$ the definition of
$(\lambda,\delta)$-hyperbolic time reduces to the first condition. Hyperbolic times were introduced in \cite{A00} for deterministic
systems
and extended in \cite{AAr03} to a random context. We recall the following results from \cite{AV10}.

\cpr\label{p.hypreball}
Given $\lambda<1$ and $\delta>0$, there exist $\delta_1, C_0>0$ only depending on $\lambda, \delta$ and $f$, such that, if $n$ is $(\lambda,\delta)$-hyperbolic time for $(\omega, x)\in \Omega\times M$, then there is a neighbourhood $V^n_{\omega}(x)$ of $x$ in $M$ s.t.:
\begin{enumerate}
\item $f^{n}_{\omega}$ maps $V^n_{\omega}(x)$ diffeomorphically onto $B(f^{n}_{\omega}(x),\delta_1)$;
\item for every $y\in V^n_{\omega}(x)$ and $1\le k\le n$ we have $\|Df^k_{\sigma^{n-k}(\omega)}(f_{\omega}^{n-k}(y))^{-1}\|\le\lambda^{k/2};$
\item for every $1\le k
\le n$ and $y, z\in V^n_{\omega}$, $$ \dist(f_{\omega}^{n-k}(y),f_{\omega}^{n-k}(z)) \le
\lambda^{k/2}\dist(f_{\omega}^{n}(y),f_{\omega}^{n}(z));$$
\item(Bounded Distortion) for any $y, z\in V^n_\omega(x)$,
 $$ \log\frac{|\det Df_{\omega}^n (y)|}{|\det Df_{\omega}^n (z)|}
 \le C_0\dist(f_{\omega}^{n}(y),f_{\omega}^{n}(z)).
 $$

\end{enumerate}
\fpr
We call the sets $ V^n_\omega$ as \emph{hyperbolic
pre-balls} and $ B(f^{n}_{\omega}(x),\delta_1)=f_\w^{n}(V^n_\omega) $
\emph{hyperbolic balls}. We recall now the random version of the positive frequency for points admitting hyperbolic times.
\cpr\label{l:hyperbolic2}
    There exist $0<\lambda<1$, $\delta>0$ and \( 0<\zeta\le 1 \) such that for every $\omega\in\Omega$ and every $x\in M$ with $\mathcal E_{\omega}(x)\le n$ and $\mathcal R_{\omega}(x)\le n$,  there exist $(\lambda,\delta)$-hyperbolic times $1\le n_1<\dots<n_l\le n$ for $(\omega, x)$ with $l\ge\zeta n$.
\fpr

For technical reasons, more precisely in Item \eqref{l.in V_n} of Lemma~\ref{l.PP}, we shall take $\delta'_1=\frac{\delta_1}{12}>0$ and consider $V'^n_{\w}(x)$ the part of $V^n_\omega(x)$ which is sent by $f^n_{\w}$ onto $B(f^n_{\w}(x),\delta'_1)$. The sets $V'^n_{\w}(x)$  will also be called hyperbolic pre-balls.
The next lemma is an immediate consequence from \cite[Lemma 2.4, 2.5]{ALP05}; see also \cite[Lemma 4.14]{AV10}.

\cle\label{l.N0q}
There are $\delta_0>0$, a point $p\in M$  and $N_0\in\NN$ s.t., if $\epsilon$ is sufficiently small, for any hyperbolic pre-ball $V'^n_{\w}(x)$ and every $\omega\in\supp(P)$ there exists $0\le m\le N_0$ for which $\Delta=B(p,\delta_0)\subset f^{n+m}_{\w}(V'^n_{\w}(x))$ and $f^{n+m}_{\w}(V'^n_{\w}(x))$ is disjointed from the critical set $\mathcal{C}$. Moreover, there are $D_0$, $K_0$ such that, for every $\omega\in\supp(P)$ we have
\begin{enumerate}
\item for each $x, y\in V$
$$
\log\left|\frac{\det Df_{\omega}^m(x)}{\det Df_{\omega}^m(y)}\right|\leq D_0
\dist(f_{\omega}^m(x), f_{\omega}^m(y));
$$
\item for each $0\leq j\leq m$
and for all $x\in f_{\omega}^{j}(V)$ we have
$$
K_{0}^{-1}\leq \|Df_{\omega}^{j}(x)\|, \|(Df_{\omega}^{j}(x))^{-1}\|, |\det
Df_{\omega}^{j}(x)|\leq K_0;$$
in particular  \( f_{\omega}^{j}(V)\cap \cc = \emptyset \) .
\end{enumerate}

\fle

In the following, we fix the two disks centered at $p$
$$\Delta^0=\Delta=
B(p,\delta_{0}){\quad\text{and}\quad}  \Delta^1= B(p,2\delta_{0}).
$$
We actually need $\Delta^1=B(p,2\delta_0)\subset f^{n+m}_{\w}(V'^n_{\w}(x))$.

\subsection {The auxiliary partition}
We construct the random Markov partition $\mathcal P_{\omega}$ on the reference disk $\Delta$ found in the previous section, for each $\omega\in\Omega$. Basically, it is a random version of the construction in \cite{AL13}. For $\omega\in \Omega$, $n\ge 1$, we define
 $$
 H^n_{\omega}=\{x\in M\colon \text{ $n$ is a $(\lambda,\delta)$-hyperbolic time for
 $(\omega, x)$ }\}.
 $$
We set $\Delta^c=M\setminus\Delta$. Given a point $x \in H^n_{\w}\cap\Delta$, there is a hyperbolic pre-ball $V'^n_{\w}(x)$, and for $0\le m\le N_0$ as in Lemma~\ref{l.N0q}, we define
\begin{equation}\label{candidate}
U_{n,m }^{i, x} = (f^{n+m}_\w|_{V'^{n}_{\omega}(x)})^{-1}(\Delta^{i}),\quad i=0,1.
\end{equation}
These $U^{0,x}_{n,m}$ are the candidates for $\cp_\w$ with uniform expansion and
bounded distortion of $f^{n+m}_\w$. Notice that the recurrence time is given by $$R_\w(x)=n+m\quad \text{for}\quad x\in U^{0,x}_{n,m}.$$
We introduce sets $\Delta^n_{\omega}$ and $S^n_{\omega}$: $\Delta^n_{\omega}$ is the part of $\Delta$ that has not been chosen until time $n$;  $S^n_{\omega}$, the
\emph{satellite set}, is to make sure we may collect the remaining part of the hyperbolic pre-ball when the $n$-step's elements of partition have been taken, more precisely, to get $$H^n_{\w}\cap\Delta\subset S^n_{\omega}\bigcup \{R_{\omega}\leq n+N_0\}.$$
 At each step of the algorithm there is a unique hyperbolic time and possibly several return times.
For $k\ge n$, we construct the \emph{annulus} around the element $U^n_{\omega}=U_{n,m}^{0,x}$
\begin{equation}\label{eq.Annulus}
A^k_{\omega}(U^n_{\omega})=\{y\in V^n_{\omega}(x):0\le\dist(f_{\omega}^{R_\w(U^n_{\omega})}(y),\Delta)\le \delta_0\lambda^{\frac{k-n}{2}}\}.
\end{equation}
Obviously $$A^n_{\omega}(U^n_{\omega})\cup U^n_{\omega}=U_{n,m}^{1,x}.$$

\subsubsection*{First step of induction}
Given the initial time \(R_0\in\mathbb{N} \) and consider the dynamics after time
$R_0$ (can be taken independent of $\omega$); to be determined in Section
 \ref{sub.ebd} this paragraph we omit $\epsilon$ in $R_0$. There are finitely many points
$I_{R_0}=\{z_1,\ldots,z_{N_{R_0}}\}\in H^{R_0}_{\omega}\cap\Delta$ such that
$$H^{R_0}_{\omega}\cap \Delta\subset V^{'R_0}_{\omega}(z_1)\cup\dots\cup
V^{'R_0}_{\omega}(z_{N_{R_0}}).$$
We take a maximal family of pairwise disjoint sets of type (\ref{candidate}) contained in
$\Delta$,
$$
\{U_{R_0,m_0}^{1,x_0}, U_{R_0,m_1}^{1,x_1},\ldots,
U_{R_0,m_{k_{R_0}}}^{1,x_{k_{R_0}}}\},
$$
where $\{x_0,\ldots x_{k_{R_0}}\}\subset I_{R_0}$.

And set
$$
\mathcal U^{R_0}_{\omega}=\{U_{R_0,m_0}^{0,x_0},
U_{R_0,m_1}^{0,x_1},\ldots,
U_{R_0,m_{k_{R_0}}}^{0,x_{k_{R_0}}}\}
$$
Now we get the elements of the partition $\cp_\w$ at $R_0$-step.

The recurrence time is $R_\w(U^{0,x_i}_{R_0,m_i})=R_0+m_i$ with
$0\leq i \leq k_{R_0}$.
Recalling \eqref{eq.Annulus}, we define $$A^{R_0}_{\omega}(\mathcal U^{R_0}_{\omega})=\bigcup_{U_\w\in\mathcal U^{R_0}_{\omega}}A^{R_0}_{\omega}(U_\w).$$
We pay attention to the sets $\{U^{1,z}_{R_0,m}: z\in
I_{R_0},0\le m\le N_0\}$ which intersect $\mathcal U^{R_0}_{\omega}\cup A^{R_0}_{\omega}(\mathcal U^{R_0}_{\omega})$ or $\Delta^c$.
Given $U_\w\in \mathcal U^{R_0}_{\omega}$, for each $0\le m\le N_0$,
we define
\begin{equation*}
    I_{R_0}^m(U_\w)=\left\{x \in I_{R_0}: U^{1,x}_{R_0,m}\cap (U_\w\cup A^{R_0}_{\omega}(U_\w)) \neq\emptyset\right\},
\end{equation*}
the \emph{$R_0$-satellite} around $U_\w$ is
\begin{equation*}
    S^{R_0}_{\omega}(U_\w)=\bigcup_{m=0}^{N_0}\bigcup_{x\in
    I_{R_0}^m(U_\w)}V^{'R_0}_{\omega}(x)\cap (\Delta\setminus
    U_\w),
\end{equation*}
The union
\begin{equation*}
S^{R_0}_{\omega}(\Delta) = \bigcup_{U_\w\in
\mathcal U^{R_0}_{\omega}}{S}^{R_0}_{\omega}(U_\w).
\end{equation*}
Similarly, the $R_0$-satellite for $\Delta^c$ is
\begin{equation*}
    S^{R_0}_{\omega}(\Delta^c)= \bigcup_{m=0}^{N_0}\bigcup_{U^{1,x}_{R_0,m}\cap\Delta^c\neq\emptyset}V^{'R_0}_{\omega}(x)\cap \Delta.
\end{equation*}
We will show in the general step, the  volume of
$S^{R_0}_{\omega}(\Delta^c)$ is exponentially small.
The `global' $R_0$-satellite is
\begin{equation*}
S^{R_0}_{\omega} = \bigcup_{U_\w\in \mathcal U^{R_0}_{\omega}}{S}^{R_0}_{\omega}(U_\w) \cup
S^{R_0}_{\omega}(\Delta^c).
\end{equation*}
The remaining portion at step $R_0$ is
\begin{equation*} \Delta^{R_0}_{\omega} = \Delta\setminus
\mathcal U^{R_0}_{\omega}.
\end{equation*}
Clearly,
\begin{equation*}
H^{R_0}_{\omega}\cap\Delta\subset S^{R_0}_{\omega}\cup \mathcal U^{R_0}_{\omega}.
\end{equation*}

\subsubsection*{General step of induction}
The general step of the construction follows the ideas in the first step with
minor modifications. We assume $U^j_{\w}, S^j_{\w}, A^j_{\w}, \Delta^j_{\w},\{R_{\w}=j+m\}$ are defined for all $0\le j\le n-1$.  As before, there is a finite set of points
$I_{n}=\{z_1,\ldots,z_{N_{n}}\}\in  H_\w^n\cap\Delta$ such that
$$ H^n_{\omega}\cap\Delta\subset
V'^n_{\omega}(z_1)\cup\dots\cup V'^n_{\omega}(z_{N_{n}}).$$ We get $\mathcal U^i_{\omega}, A^i_{\omega}$ and $S^i_{\omega}$ for $i\le n-1$.
Assuming
$$\mathcal U^{\ell}_{\omega}=\{U_{\ell,m_0}^{0,x_0}, U_{\ell,m_1}^{0,x_1},\ldots,
U_{\ell,m_{k_{\ell}}}^{0,x_{k_{\ell}}}\}
$$
for $ R_0\le\ell\le n-1$, let
$$A^n_{\omega}(\mathcal U^{\ell}_{\omega})=\cup_{U_\w\in\mathcal U^{\ell}_{\omega}}A^n_{\omega}(U_\w).$$
We get a maximal family of pairwise disjoint sets of type \eqref{candidate} contained in
$\Delta^{n-1}_{\omega}$,
$$ \{U_{n,m_0}^{1,x_0},
U_{n,m_1}^{1,x_1},\ldots, U_{n,m_{k_n}}^{1,x_{k_n}}\}$$ where $\{x_0,\ldots,x_{k_n}\}\subset I_n$,
satisfying
$$ U_{n,m}^{1,x_i}\cap\left(
\cup_{\ell=R_0}^{n-1}\{A^n_{\omega}(\mathcal U^{\ell}_{\omega})\cup\mathcal U^{\ell}_{\omega}\}\right)=\emptyset,\quad
i=0,\dots,k_n,$$
and define
$$\mathcal U^n_{\omega}=\{U_{n,m_0}^{0,x_0},
U_{n,m_1}^{0,x_1},\ldots,
U_{n,m_{k_{n}}}^{0,x_{k_{n}}}\}.$$
They are the $n$-step's elements of the partition $\cp_{\omega}$.

For $0\leq i \leq \ell_n$, $x\in U^{0,x_i}_{n,m_i}$, $R_{\omega}(x)=n+m_i$.
Given $U_\w\in \mathcal U^{R_0}_{\omega}\cup\cdots\cup \mathcal U^n_{\omega}$, $ 0\leq
m\leq N_0 $,
\begin{equation*}
    I_{n}^m(U_\w)=\left\{x \in I_{n}: U^{1,x}_{n,m}\cap
    (U_\w\cup A^n_{\omega}(U_\w))
\neq\emptyset\right\},
\end{equation*}
  define
\begin{equation*}
S^n_{\omega}(U_\w)= \bigcup_{m=0}^{N_0}\bigcup_{x\in
    I_{n}^m(U_\w)}V'^n_{\omega}(x)\cap (\Delta\setminus U_\w)
\end{equation*}
and
\begin{equation*}
    S^n_{\omega}(\Delta)= \bigcup_{U_\w\in \mathcal U^{R_0}_{\omega}\cup\cdots\cup\,
\mathcal U^n_{\omega}}{S}^n_{\omega}(U_\w).
\end{equation*}
Similarly, the \emph{$n$-satellite} associated to $\Delta^c$ is
\begin{equation*}
    S^n_{\omega}(\Delta^c)= \bigcup_{m=0}^{N_0}\bigcup_{U^{1,x}_{n,m}\cap\Delta^c\neq\emptyset}V'^n_{\omega}(x)\cap\Delta.
\end{equation*}

\cre\label{lastremark}
By Proposition~\ref{p.hypreball},
$$
{S}^n_{\omega}({\Delta^c})\subset \{x \in \Delta:\,
\dist(x,\partial\Delta)\le 2\delta_0\lambda^{n/2}\}.
$$
So there exists $\rho>0$ such that
$
\leb ({S}^n_{\omega}({\Delta^c}))\leq \rho\lambda^{n/2}.
$
\fre

Finally we define the \emph{$n$-satellite} for $\mathcal U^{R_0}_{\omega}\cup\cdots\cup\, \mathcal U^n_{\omega}$
\begin{equation*}
{S}^n_{\omega} = S^n_{\omega}(\Delta) \cup {S}^n_{\omega}(\Delta^c)
\end{equation*}
and
\begin{equation*}
\Delta^n_{\omega}= \Delta\setminus \bigcup_{i=R_0}^{n}\mathcal U^i_{\omega}.
\end{equation*}
Obviously
\begin{equation}\label{Hdel_n}  H^n_{\omega}\cap\Delta\subset
{S}^n_{\omega}\cup\bigcup_{i=R_0}^{n}\mathcal U^i_{\omega}.
\end{equation}

\subsection {Expansion, bounded distortion and uniformity}\label{sub.ebd}
The return time $R_{\omega}$ for an element $U_{\omega}$
of the partition $\cp_{\omega}$ of $\Delta$ is made by a hyperbolic time $n$ plus $m\leq N_{0}$. We know $f_{\omega}^{n+m}(V'^n_{\omega})$ covers $\Delta$ completely. Then by Proposition~\ref{p.hypreball} and Lemma~\ref{l.N0q},

\begin{eqnarray*}
  \|Df_{\omega}^{n+m}(x)^{-1}\| &\le& \|Df_{\sigma^n(\w)}^{m}(f_{\omega}^{n}(x))^{-1}\|.
\|Df_{\omega}^{n}(x)^{-1}\| \\
   &\le &K_{0}\lambda^{n/2}\\
   &\le& K_{0}\lambda^{(R_0-N_{0})/2}.
\end{eqnarray*}
If we take $R_0$ sufficiently large, this is smaller than some $\kappa<1$. We also need to show that there exists a constant
$K > 0$ such that for any $x, y\in U_{\omega}$
with return time $R_{\omega}$, we have
$$\log\left|\frac{\det Df_{\omega}^{R_{\omega}}(x)}{\det Df_{\omega}^{R_{\omega}}(y)}\right|\leq K \dist(f_{\omega}^{R_{\omega}}(x),
f_{\omega}^{R_{\omega}}(y)).$$ By Proposition~\ref{p.hypreball} and
Lemma~\ref{l.N0q}, we choose $K = D_{0}+
C_{0}K_{0 }$.

Since $K_0$, $C_0$, $D_0$ and $N_0$ are independent
of $\omega$ then $\kappa$ and $K$ could be taken the same for all $\w$, leading us to condition (U$_2$). Moreover, by the continuity of $\Phi$ in the random perturbation
$\{\Phi,\{\theta_\epsilon\}_{\epsilon>0}\}$, the algorithm provides partitions such that for any two realizations $\w$, ${\w'}$ in
 $\Omega$ and any natural number $\hat N$, the Lebesgue measure of the symmetric difference of
sets $\{R_{\w}=j\}$ and $\{R_{\w'}=j\}$, for
 $j=1,\ldots, \hat N$, is smaller than any given ${\xi}>0$, as long as we take $\epsilon$ sufficiently
 small. Since we do not assume a particular behavior for the decay of the tail set for the deterministic dynamics given by $f$, we are not able to conclude the (stretched) exponential decay of the corresponding return times. However we can construct a partition for
 for the original dynamics $f$ (given by the realization $\w^*$) and consider it as reference to construct the elements of each partition $\mathcal P_\w$ with return time lower than $\hat N$,
 obtaining condition (U$_1$). This is of great utility to ensure condition ($\star$); see Remark \ref{r gcd}.

\subsection{Tail set estimates}\label{s.tail}
In this section we will show that if the tail set decays (stretched) exponentially fast, then the tail of the recurrence times decays (stretched) exponentially fast too.
More precisely, given a local unstable disk $\Delta\subset M$ and constants $\gamma>0$ and $0<\upsilon\le 1$, there is $\gamma_1>0$ such that
\begin{equation}\label{eq.goaltail}
\leb\{\Gamma^n_{\w} > n\} =\mathcal O(e^{-\gamma n^{\upsilon}})\quad\Rightarrow\quad
\leb\{R_{\w}>n\}\leq\mathcal {O}(e^{-\gamma_1n^{\upsilon}}).
\end{equation}
This is case $(i)$ of Theorem \ref{random Markov structure}. Case $(ii)$ follows in the same way.

Before the key proof (proof for Proposition \ref{Z}). We state some lemmas and notations for preparing.
To simplify the notation, we avoid the superscript 0 in $U^{0,x}_{n,m}$.
The next lemma and proposition are the random versions of \cite[Lemma 3.5, Proposition 3.6]{AL13}.

\cle
\begin{enumerate}
\item There is $C_5>0$, for any $n\ge R_0$, $0\leq m\leq N_0$ and finitely many $\{x_1,\dots ,x_N\}\in I_n$ satisfying
$U_{n,m}^{x_i}=U_{n,m}^{x_1}$ ($1\le i\le N$), we get
\begin{equation*}
    \leb\left(\bigcup_{i=1}^{N} V'^n_{\omega}(x_i)\right)\leq C_5\leb(
    U^{x_1}_{n,m}).
   \end{equation*}
\item There is $C_6> 0$, for $k\ge R_0$, $U_{\omega}\in \mathcal U^k_{\omega}$ and $0\le m\le N_0$, any $n\ge k$, we have
\begin{equation*}
\leb \left(\bigcup_{x \in I_n^m(U_{\omega})}U_{n,m}^x\right)\leq
C_6\lambda^{\frac{n-k}{2}} \leb(U_{\omega}).
\end{equation*}
\end{enumerate}
\fle

 \cpr \label{d.prop.Sn}
There is $C_7>0$ such that $\forall\, U_{\omega}\in\mathcal U_{\omega}^k$, and $n\geq
k$, we get
\begin{equation*}
\leb(S^{n}_{\omega}(U_{\omega})) <C_7\lambda^{\frac{n-k}{2}}\leb(U_{\omega}).
\end{equation*} \fpr

\cde
Given $k\ge R_0$ and  $U^x_{k,m}\in \mathcal U_{\omega}^k$, $x\in \Delta$ and $0\le m\le N_0$, for $n\ge k$ we define
 $$B_n^k(x)=S^{n}_{\omega}(U_{k,m}^{x})\cup
U_{k,m}^{x}\quad\text{and}\quad t(B_n^k(x))=k.$$ Here
$k$ and $n$ are hyperbolic times for points in $\Delta$. We call $U^x_{k,m}$ the \emph{core} of $B_n^k(x)$
and sign it $C(B_n^k(x))$.
\fde
From Lemma~\ref{d.prop.Sn} we easily get:  $\forall n\ge k$ and $x$, we have
 $$\leb(B_n^k(x))\le
(C_7+1)\leb(C(B_n^k(x))).$$
The dependence of $\delta_1'$ on $\delta_1$ is clarified in the next lemma. The proof is similar with \cite[Lemma 3.9, 3.10]{AL13}.

\begin{Lemma}\label{l.PP}
\begin{enumerate}
\item\label{l.in V_n}
If $\delta_1'>0$ is sufficiently small (only depending on $\delta_1$), for all $k'\ge k\ge R_0$, $n\ge k$, $n'\ge k'$ and $B_n^k(x)\cap
B_{n'}^{k'}(y)\neq\emptyset$, we have $$C(B_n^k(x))\cup
C(B_{n'}^{k'}(y))\subset V^k_{\omega}(x);$$

\item
there exists $P\ge N_0$ such that for all $R_0\leq t_1\le t_2$,
$B_{t_2+P}^{t_2}(y)\cap B_{t_2+P}^{t_1}(x)=\emptyset.$
\end{enumerate}
\end{Lemma}
Now we come to the core of this section: to show \eqref{eq.goaltail}, i.e. Theorem \ref{random Markov structure}. Recalling Remark \ref{lastremark}, similarly, there exists a constant $\rho >0$ such
that for all $n\in\mathbb N$
\begin{equation}\label{bdy2}
\leb\{x :
\dist(x,\partial\Delta)\leq 2\delta_0\lambda^{\frac{\theta n}{4}}\}\leq
\rho\lambda^{\frac{n}{2}}.
\end{equation} Recalling $\Delta^n_{\omega}$ is the
complement part at step $n$, $\theta$ is defined in
Proposition~\ref{l:hyperbolic2}. We will show $\leb(\Delta^n_{\omega})$ decays (stretched) exponentially.
That is enough to conclude the proof since $\leb(\Gamma_\w^n)$ is (stretched) exponentially small and
$\leb(\{x:\dist(x,\partial\Delta)\leq 2\delta_0\lambda^{\frac{\theta n}{4}}\})$
decays exponentially as in \eqref{bdy2}.

Take $x\in\Delta^n_{\omega}$, suppose $x\notin \Gamma_\w^n\cup\{x:\dist(x,\partial\Delta)\leq 2\delta_0\lambda^{\frac{\theta n}{4}}\}$.
By Proposition \ref{l:hyperbolic2}, for $n$ large, $x$ has at least $\theta n$
hyperbolic times between $1$ and $n$, such that we have $\frac{\theta n}{2} \le t_1<\dots< t_k\le n$, $k\ge\frac{{\theta}n}{2}$. We get $x\in H^{t_i}_{\omega}\cap\Delta$ for $1\le i\le k$.
Recalling \eqref{Hdel_n},
\begin{equation*}
 H^{t_i}_{\omega}\cap\Delta\subset S^{t_i}_{\omega}\cup\bigcup_{j=R_0}^{t_i}\mathcal U^j_{\omega}, \quad \text{for } 1\le i\le k.
\end{equation*}
If $x\notin S^{t_i}_{\omega}$, $x\in\cup_{j=R_0}^{t_i}\mathcal U^j_{\omega}$ such that
$x\notin\Delta^n_{\omega}$. That is a contradiction. So $x\in S^{t_i}_{\omega}$.
Since $x\in\{x\in\Delta:\dist(x,\partial\Delta)> 2\delta_0\lambda^{\frac{{\theta}n}{4}}\}$,
$x\in H^{t_i}_{\omega}\cap \{x\in\Delta:
\dist(x,\partial\Delta)> 2\delta_0\lambda^{t_i/2} \},$ for $1\le i\le k.$
With Remark \ref{lastremark}, we obtain
$x\notin S^{t_i}_{\omega}(\Delta^c)$.
Consequently, $x\in S^{t_i}_{\omega}(\Delta),$ for $i=1,\dots,k.$ We simply
take $k=\frac{{\theta}n}{2}$. Thus, $x$ is contained in
$$Z_\w\left(\frac{\theta n}{2},n\right):=\bigg\{x:\exists t_1<\ldots<t_{\frac{\theta n}{2}}\leq
n, x\in \bigcap_{i=1}^{\frac{\theta
n}{2}}{S^{t_i}_{\omega}(\Delta)}\bigg\}\cap\Delta^n_{\omega}.$$
So we have
$$\Delta^n_{\omega} \subset \Gamma_{\w}^n\cup \{x\in\Delta: \dist(x,\partial\Delta)
\leq 2\delta_0\lambda^{\frac{{\theta}n}{4}}\}\cup Z_\w({\theta n}/2,n).$$
See the first set in the union above decays exponentially fast from the assumption of Theorem~\ref{random Markov structure}; the second set in the union above is exponentially small by \eqref{bdy2}. In the following, we only need to show the measure of
$Z_\w({\theta n}/2,n)$ is exponentially small. That is Proposition~\ref{Z} .

Observe that if we have shown there exist $C_1, \gamma_1>0$ such that
$$\leb(\Delta^n_{\omega})\leq C_1e^{-\gamma_1 n^{\upsilon}},
$$
then, for any large integer $n$, we have $\mathcal
{R}^n_\w=\{R_\w>n\}\subset\Delta^{n-N_0}_{\omega}$, and so
\begin{equation}\label{eq.tailgoal}
\leb(R_\w>n)\leq\leb(\Delta^{n-N_0}_{\omega})=C_1e^{-\gamma_1{(n-N_0)}^{\upsilon}}=C_1e^{-\gamma_1 n^{\upsilon}}.
\end{equation}
We show the set of points which are contained in finitely many satellite sets and have not been chosen yet has a measure exponentially small.
\begin{Proposition}\label{Z}
For $k,N\in\ZZ_+$,
\begin{eqnarray*}
Z_\w(k,N)&=& \bigg\{x:\exists t_1<\ldots<t_k\leq N, x\in
\bigcap_{i=1}^k{S^{t_i}_{\omega}(\Delta)}\cap\Delta^N_{\omega}\bigg\}.
\end{eqnarray*}
There are $D_3>0$ and $\lambda_3<1$, for all $N$ and $1\le
k\leq N$,
$$\leb(Z_\w(k,N)) \le D_3\lambda_3^k\leb(\Delta).$$\end{Proposition}

In order to prove this result we need several pre-lemmas in the sequel.
We fix some integer $P'\ge P$ (see $P$ in
Lemma~\ref{l.PP}; see the proof of
Proposition~\ref{Z}.
In the following, for some $t_i,x$, $m_i\le P'$ we denote $B_i=B_{t_i+m_i}^{t_i}(x)$. The proof of the next two lemmas may be found in \cite[Lemma 3.14, 3.15]{AL13}.
\begin{Lemma}\label{Z_1}
Set
\begin{eqnarray*}
Z_\w^1(k,N) &=& \bigg\{ x: \exists t_1<\ldots<t_k\leq N, m_1,\ldots,m_q<P',\\
& & x\in
S^{t_1+m_1}_{\omega}(\mathcal U^{t_1}_{\omega})\cap\ldots\cap
S^{t_k+m_k}_{\omega}(\mathcal U^{t_k}_{\omega})\cap\Delta^N_{\omega} \bigg\}.
\end{eqnarray*}

 There are constants $D_1>0$ and $\lambda_2<1$
(both independent of $P'$) such that, for all $N$ and $1\leq k \leq
N$,
$$\leb(Z_\w^1(k,N))\leq D_1{\lambda_2^k}\leb(\Delta).$$
\end{Lemma}

\begin{Lemma}\label{Z_2}
Given $B_1=B_{t_1}^{t_1}(x_1)$, let
\begin{eqnarray*}
Z_\w^2(n_1,\ldots,n_k,B_1) &=& \bigg\{x:\exists\,\, t_2,\ldots,t_k
\,\, with \,\,t_1<\ldots<t_k ; n_1,\ldots,n_k>P; and \,\,x_2,\ldots,x_k,\\
& & \bigg. s.t. \,\,x\in
\bigcap_{i=1}^{k}B_{t_i+n_i}^{t_i}(x_i)\cap\Delta^N_{\omega}\bigg\}.
\end{eqnarray*}
Then, there is  $D_2>0$ (independent of $B_1,n_1,\ldots,n_k$)
such that for $n_1,\ldots,n_k>P$,
$$\leb(Z_\w^2(n_1,\ldots,n_k,B_1))\leq
D_2(D_2\lambda^{n_1 /2})\dots(D_2\lambda^{n_k /2})\leb (C(B_1)).
$$
\end{Lemma}

Then we complete the proof of the metric estimates.

\begin{proof}[Proof of Proposition~\ref{Z}]
Take $P'\geq P$ (recall $P$ in Lemma~\ref{l.PP}) such that
$$\lambda^{1/2}+D_1\lambda^{P'/2}<1.$$
Let $x\in Z_\w(k,N)$, we have all the instants $u_i$ for which $x\in S^{u_i+n_i}_\w(U_{u_i,m}^{y})$ with $n_i\geq P'$, ordered as $u_1<\ldots<u_p$. Then $x\in
Z_\w^2(n_1,\ldots,n_p,B_1)$ for some $B_1$. If $\sum_{i=1}^{p}{n_i}\geq
k/2$, we are done. Otherwise, we have $\sum_{i=1}^{p}{n_i}< k/2$ and $p<k/2P'$.  Then $v_1<\ldots<v_q$ be the other instants for which $x\in S_\w^{v_i+m_i}(U_{v_i,\tilde{m}}^{z})$, where
$m_1,\ldots,m_q<P'$. Obviously $p+q\geq k$, so that $q\ge\frac{(2P'-1)k}{2P'}\geq\frac{k}{2P'}$, where $P'>1$. So $P'q\geq\frac k2$.
Thus we obtain
$$Z_\w(k,N)\subset
\bigcup_{B_1}\bigcup_{\mycom{n_1,\ldots,n_p\geq
P',}{\sum{n_i}\geq\frac{k}{2}}}Z_\w^2(n_1,\ldots,n_p,B_1)\cup{Z_\w^1\left(\frac{k}{2P'},N\right)}.
$$
By Lemma~\ref{Z_1} and ~\ref{Z_2}, we obtain
$$\leb(Z_\w(k,N))\leq \sum_{B_1}\sum_{\mycom{n_1,\ldots,n_p\geq
P',}{\sum{n_i}\geq\frac{k}{2}}} D_2(D_2\lambda^{n_1
/2})\ldots(D_2\lambda^{n_p
/2})\leb(C(B_1))+D_1\lambda_2^{\frac{k}{2P'}}\leb(\Delta).
$$
We know $\sum_{B_1}\leb(C(B_1))\leq \leb(\Delta)<\infty$ as the cores $C(B_1)$ are pairwise disjoint. There are constants $D_4>0$ and $\lambda_4<1$
such that
$$\sum_{\mycom{n_1,\ldots,n_p\geq
P',}{\sum{n_i}=n}}(D_2\lambda^{n_1
/2})\ldots(D_2\lambda^{n_p /2})\leq
D_4\lambda_4^n.
$$
Sum over $n\geq k/2$ and $B_1$, we obtain constants $D_3>0$, $\lambda_3<1$ such that
$$\leb(Z_\w(k,N))\leq D_3\lambda_3^k\leb(\Delta).$$
\end{proof}

\section{Decay of correlations on random Young towers}\label{Decay of correlations for random GMY}

In this section we prove Theorem \ref{dacayrate}. We start by compiling in \S\ref{induced mix} and \S\ref{CEDC} some definitions and key results from \cite{BBM02}, which are randomised
versions of that in \cite{Y99}.  Then, in \S\ref{rtst} we transpose the hypotheses on the return times to  the estimates on the (joint) return times. Finally, in \S\ref{stdc} we give the estimates on the induced decay of correlations. We will focus on the future time results, being that the results for the past correlations are the recycling of the arguments for the future ones, as noticed in \cite[\S6]{BBM02}.

\subsection{Mixing}\label{induced mix}
Recall the abstract setting $\hat\Delta=\{\Delta_\w\}_\w$ with the dynamics of the fibered map $F=\{F_\w\}_\w$ that we call the \emph{induced skew product}. We concern now to the mixing properties of the induced skew product with respect to a measure $\nu$ whose disintegration  $d\nu(\omega,x)=d\nu_{\omega}(x)
dP(\omega)$ is given by the family $\{\nu_\w\}_\w$ of
sample measures constructed at Theorem \ref{exist.mu.induced}
(for $A\in\mathcal B$, we have $\nu(A)=\int \nu_\w (A_\w) dP$).
For $n\geq1$ we set $F_\w^n$ for the compositions $F_{\sigma^{n-1}(\w)}\circ\cdots\circ F_\w$ and
also $F^{-n}(\mathcal B)$ for the family $\{(F_\w^n)^{-1}(\mathcal B_{\sigma^n(\w)})\}_\w$.
 Let $L^2(\nu)$ denote the space of functions $\phi=\{\phi_\w\}_\w:\hat\Delta\to\mathbb R$ such that $\phi_\w\in L^2(\mathcal B_\w,\nu_\w)$ for $P$ a.e. $\w$, and $\int_\Omega\int_{\Delta_\w}|\phi_\w|^2\,d\nu_\w dP< \infty$.

\cde We say that the random skew product $(F,\nu)$ is
\begin{enumerate}
\item\emph{exact} if  each $B\in\mathcal B$ belonging to $F^{-n}(\mathcal B)$ for all $n\geq 0$ is trivial (i.e., for almost every $\w$, either $\nu_\w(B)=0$ or $\nu_\w(B)=1$);
\item \emph{mixing} if for all $\varphi, \psi\in L^2(\nu)$, $$\lim_{n\to+\infty}\left|\int_\Omega\int_{\Delta_\w}(\varphi_{\sigma^n(\w)}\circ F_\w^n)\psi_\w\,d\nu_\w d P-\int_\Omega\int_{\Delta_\w}\varphi_{\w}\,d\nu_\w d P\int_\Omega\int_{\Delta_\w}\psi_\w\,d\nu_\w d P \right|=0.$$
\end{enumerate}
\fde

\cpr
If $(F,\nu)$ is exact then it is mixing.
\fpr

However, for the exactness (and mixing) of $(F,\nu)$ we  need to assume the following:
\begin{enumerate}
\item[($\star$)] there are $L_0(\epsilon)$ and $t_i\in \ZZ_+, 1\leq i \leq L_0$, with $g.c.d.\{t_i\}=1$ such that for $P$ a.e. $\w$,
\begin{equation*}
\leb(\{R_\w=t_i\})>0.
\end{equation*}
\end{enumerate}

 \cpr
 $(F,\nu)$ is exact (and thus mixing).
  \fpr

 \subsubsection{A remark on the return times}\label{r gcd}
One knows that there exists a GMY structure for $f$. If $g.c.d.\{R_f\}=1$ then there are $L_0$ and $t_i\in \ZZ_+, 1\leq i \leq L_0$, with $g.c.d.\{t_i\}=1$ such that $\leb(\{R_f=t_i\})>0$. Hence, condition (U$_1$) ensures that ($\star$) hold,
 just considering $\hat N=t_{L_0}$ and $\xi$ sufficiently small (only depending on $\{R_f\}$). On the other hand, if $g.c.d.\{R_f\}=q>1$ then, the previous
 partition related to $f$ also provides a partition for $f^q$, with $R_{f^q}=R_f/q$ and, in this case, $g.c.d.\{R_{f^q}\}=1$. Thus, we look then for the $q$th iterate $f_\w^q$ for the random systems. We consider the partitions $\{\bar R_\w=k\}=\{R_\w=q\cdot k\}$ and the fibered dynamics $\bar F$ ("$= F^q$") in the \emph{new} towers $\bar \Delta=\{\bar\Delta_\w\}_\w=\{\cup_{k=0}^\infty \Delta_{\w,qk}\}_\w$. Similarly to Theorem \ref{exist.mu.induced} we may obtain a family of probability measures $\{\bar\nu_\w\}$ that constitutes a disintegration of an $\bar F$-invariant probability measure $\bar \nu$, satisfying $({\bar F}_\w)_*\bar\nu_\w=\bar\nu_{\sigma^q(\w)}$. Their projections $\bar\mu_\w=\bar\pi_*\bar\nu_\w$ are absolutely continuous probability measures for which $(f_\w^q)_*\bar\mu_\w=\bar\mu_{\sigma^q(\w)}$, and the measure $\bar\mu=\{\bar\mu_\w\}$ is invariant for the power $S^q$ of the skew product. Once again, by  (U$_1$) the condition ($\star$) hold provided $\epsilon$ is sufficiently small. In this case, our proofs yield the (stretched) exponential decay of correlations for the $qth$ iterate of the perturbed system.
If we are able to guarantee a GMY induced map for $f$ with $g.c.d.\{R_f\}=1$, then the main results hold with $q=1$. For simplicity in the exposition, henceforth we will always assume that ($\star$) hold.

\subsection{Converging to equilibrium}\label{CEDC}

\subsubsection{Stopping times and joint returns}
Let us define the random variable $$V^{\ell}_{\w}=\leb(\Delta_{\w,0}\cap (F_{\w}^\ell)^{-1}(\Delta_{\sigma^\ell{(\w)},0})).$$
From condition $(\star)$ we may consider $\ell_0\in\mathbb N$ such that, for $P$ a.e. $\w$ and $\ell\geq\ell_0$, we have $V^{\ell}_{\w}>0$. For $\w\in\Omega$ and $(x,x')\in \Delta_\w\times\Delta_\w$ we introduce the \emph{stopping times} $\tau_\w^i(x,x')$ as follows:
\begin{eqnarray*}
\tau_\w^1(x,x')&=&\inf\{n\geq \ell_0: F_\w^n(x)\in\Delta_{\sigma^n(\w),0}\},\\
\tau_\w^2(x,x')&=&\inf\{n\geq \ell_0+\tau_\w^1(x,x'): F_\w^n(x')\in\Delta_{\sigma^n(\w),0}\},\\
\tau_\w^3(x,x')&=&\inf\{n\geq \ell_0+\tau_\w^2(x,x'): F_\w^n(x)\in\Delta_{\sigma^n(\w),0}\},\\
&\vdots&
\end{eqnarray*}
with the action alternating between $x$ and $x'$.
We define then the \emph{joint return time} $T_\w(x,x')$ to be the smallest integer $\tau_\w^i=\tau_\w^i(x,x')\geq\ell_0$ such that $(F_\w^{\tau_\w^i}(x),F_\w^{\tau_\w^i}(x'))$ belongs to $\Delta_{\sigma^{\tau_\w^{i}}(\w),0}\times\Delta_{\sigma^{\tau_\w^{i}}(\w),0}$, with $i\geq2$.
Note that  $\tau_\w^i -\tau_\w^{i-1}\geq \ell_0$ and $T_\w(x,x')\geq 2\ell_0$. Given $\w$ and $j\geq 1$, we consider also the partition $\xi_\w^j$ of $\Delta_\w\times\Delta_\w$
 into maximal subsets on which $\tau_\w^i$ is constant for all $1\leq i \leq j$.

We should notice that even under hypothesis of uniform decay of the tail sets we are not able to guarantee an uniform control of random variables $V_\w^\ell$. Indeed, the induced sample measures $\{\nu_\w\}_\w$ have densities uniformly bounded from above by $K_1$ but not from below (see \cite{BBM02} and, in particular,  its \emph{corrigendum}). In view of this we cannot exploit the mixing properties of the induced skew product, and we are endorsed to a large deviation arguments. As we will see later, this is the principal cause for successive damages on the (stretched) exponential estimates.

For $q\in\NN$ and each fixed sequence of integers $0=\tau_0<\tau_1<\ldots<\tau_q$, with $\tau_i-\tau_{i-1}\geq \ell_0$, we set
$$Q_q^{\{\tau_i\}}(\w)=\sum_{i=1}^q V^{\tau_i-\tau_{i-1}}_{\sigma^{\tau_{i-1}}(\w)}.$$

\begin{Lemma}
 There exists $\rho>0$ and $0<\varrho<1$ such that for each $q\in\NN$ and every fixed sequence of integers $0=\tau_0<\tau_1<\ldots<\tau_q$, with $\tau_i-\tau_{i-1}\geq \ell_0$, there is a set
 $M_q^{\{\tau_i\}}\subset\Omega$, with $P(M_q^{\{\tau_i\}})\leq\varrho^q$ and such that if $\w\notin M_q^{\{\tau_i\}}$ then $Q_q^{\{\tau_i\}}(\w)\geq\rho q$.
 \end{Lemma}
From now on, let $\lambda,\lambda'$ be absolutely continuous probability measures on $\Delta$ with densities $\varphi,\varphi' \in \mathcal F_\beta^+$, and set $\Lambda=\lambda\times\lambda'$. In particular, we could take $\lambda$ or $\lambda'$ as $\nu$. Let us state a lower bound for $\Lambda_\w(\{T_\w=\tau_\w^i\})$, by transposing the previous large deviation arguments for estimates on the sets on towers.

\begin{Corollary}\label{coroln4}
Assume additionally that $\varphi, \varphi'\in\mathcal L^\infty$. There exist  $C>0$, $0<\varrho<1$ and a random variable $n_4(\w)$ defined on a full $P$-measure subset of $\Omega$ such that
 \begin{equation*}
\left\{
\begin{array}{ll}
K_1\Lambda_\w(\{(x,x')\in\Delta_\w\times\Delta_\w:Q_n^{\{\tau_\w^i(x,x')\}}(\w)<\rho n\})\leq \varrho^{n/2}, & \forall n\geq n_4(\w)\\
P(\{n_4(\w)>n\})\leq  C \varrho^{n/2},& \forall n\geq 1.
      \end{array}\right.
\end{equation*}
\end{Corollary}

We give now some estimates on stopping times and joint return times.

 \begin{Lemma}\label{vis}
 If $\Gamma\in\xi_\w^i$ is such that $T_\w\vert_\Gamma>\tau_\w^{i-1}$, then letting $V^{\tau_{\w}^i-\tau_{\w}^{i-1}}_{\sigma^{\tau_{i-1}}(\w)}$ be associated to $\tau_\w^{i}(\Gamma)$,
$$\Lambda_\w(\{T_\w>\tau_\w^i\}\vert\Gamma)\leq 1- V^{\tau_{\w}^i-\tau_{\w}^{i-1}}_{\sigma^{\tau_{\w}^{i-1}}(\w)}/C$$
where $C=C(\lambda,\lambda')$ depends on the Lipschitz constants of $\varphi$ and $\varphi'$. This dependence can be removed if we consider $i\geq  i_0(\varphi,\varphi')$.
\end{Lemma}
We relate now the stopping times with the return times.
\begin{Lemma}\label{K2} For all $i,n\geq 0$ and $\Gamma\in \xi_\w^i$  we have
$$\Lambda_\w(\{\tau_\w^{i+1}-\tau_\w^i>\ell_0+n\}\vert \Gamma)\leq K_2\leb(\{R_{\sigma^{\tau_i+\ell_0}(\w)}>n\})\cdot\leb(\Delta_{\sigma^{\tau_i+\ell_0}(\w)}),$$
where $K_2=K_2(\varphi,\varphi')$ depends on the Lipschitz constants of $\varphi$ and $\varphi'$. This dependence  can be removed if we consider $i\geq  i_0(\varphi,\varphi')$.
\end{Lemma}

From now on we will assume that $\varphi,\varphi' \in \mathcal F_\beta^+\cap\mathcal L^\infty$.

\subsection{From return times to joint stopping times}\label{rtst}
In this section we relate the return times to the joint stoping times. We deal separately with the uniform and non-uniform decay of the return times, as in the hypotheses of Theorem \ref{dacayrate}. We adapt the strategy in \cite{G06} to optimise the stretched estimates when comparing to the random version of Young's work \cite{Y99} used in \cite{BBM02}. However, as we will see in \S\ref{nudrt}, in the non-uniform case there are some damages on the estimates during the transposing from return times to joint stopping times, mainly due to Lemma \ref{vis} and Corollary \ref{coroln4}.

\subsubsection{Uniform decay of return times} We consider the uniform decay of return times for the induced structure.

\begin{Lemma}\label{R to T unif}
If there exist $C_1, \gamma_1>0$ and $0<\upsilon\leq1$ such that for $P$ a.e. $\w$ we have
\begin{equation*}
\begin{array}{ll}
        \leb(\{R_\w>n\})\leq C_1e^{-\gamma_1 n^\upsilon}, & \forall\,n\geq 1,
      \end{array}
\end{equation*}
then there exist $C_i, \gamma_i$, $i=1,2$, and for $P$ a.e. $\w$ a positive integer $n_2(\w)$, such that
\begin{equation*}
\left\{
\begin{array}{ll}
       \Lambda_\w(\{T_\w>n\})\leq C_1e^{-\gamma_1 n^{\upsilon}}, & \forall\,n\geq n_2(\w) \\
         P(\{ n_2(\w)>n\})\leq C_2e^{\gamma_2 n^{\upsilon}},& \forall\,n\geq 1,
      \end{array}\right.
\end{equation*}
\end{Lemma}
\begin{proof}
Part of the proof is a randomised version of \cite[lemma 4.2]{G06}, where we must take (carefully) $L=1$, $t=\tau$, $\tau=T$ and $\mu=\Lambda$.
From Lemma \ref{K2} we may consider  $a_n=\tilde C_1 e^{-\tilde\gamma_1 n^\upsilon}$ to be so that $$\Lambda_\w(\{\tau_\w^j-\tau_\w^{j-1}=n\vert \tau_\w^1,\ldots,\tau_\w^{j-1}\})\leq a_n.$$
 We recall that the convolution $b^1\star
b^2$  of two real sequences $b^1$ and $b^2$ is given by $$(b^1 \star b^2)_n =\sum_{i+j=n}
b^1_i b^2_j.$$ When $b$ is a sequence, we also write $b^{\star \ell}$
for the sequence obtained by convolving $\ell$ times the sequence $b$ with itself. As shown in \cite{G06}, for large enough $K$ the sequence $b_n=1_{n\geq K} a_n$ satisfies
  \begin{equation*}
   (b\star b)_p \leq b_p,\,\,\,\forall p\in \NN.
  \end{equation*}
We define the measurable function $k_\w:\Delta\to\mathbb N$ as follows: if $(x,x')\in\Delta_\w\times\Delta_\w$ is such that $T_\w(x,x')=\tau_\w^i(x,x')$, then we set $k_\w(x,x')=i$.
Let $k\geq 0$ and $A \subset \{1,\ldots,k\}$. For $j\in A$, take
$n_j \geq 1$. Set
\begin{equation*}Y_\w(A,n_j)=\{(x,x')\in\Delta_\w\times\Delta_\w : k_\w(x,x') \geq \sup(A) \text{ and }
 \tau_\w^j(x,x')-\tau_\w^{j-1}(x,x')=n_j, \forall j \in A\}.
 \end{equation*}
  Conditioning
successively with respect to the different times, we get
  \begin{equation*}
  \Lambda_\w \bigl( Y_\w(A,n_j)\bigr) \leq \prod_{j\in A} \Lambda_\w(\{ \tau_\w^j-\tau_\w^{j-1}=n_j \vert
  \tau_\w^{1},\ldots,\tau_\w^{j-1}\})
  \leq \prod_{j\in A} a_{n_j}.
  \end{equation*}

For each $n\in\NN$ we  define $q(n)=\lfloor \alpha n^{\upsilon}\rfloor$, where $\alpha$ is to be determined later. Let $(x,x')$ be such that $T_\w(x,x')>n$. If $k_\w(x,x')=\ell\leq q(n)$,  let $n_j=\tau_\w^j(x,x')-\tau_\w^{j-1}(x,x')$ for $j \leq \ell$, and
$A=\{ j : n_j \geq K\}$. Thus, $(x,x') \in Y_\w(A,n_j)$ and $\sum_{j\in A} n_j \geq n/2$ if $n$ is large enough. Consequently,
  \begin{equation}
  \label{cdcd0}
  \{ (x,x') : T_\w(x,x') > n\}
  \subset \{ k_\w(x,x')> q(n)\} \cup
  \bigcup_{A \subset \{1,\ldots,q(n)\}}
  \bigcup_{\substack{n_j \geq K \\ \sum_A n_j \geq n/2}} Y_\w(A, n_j).
  \end{equation}

Following the estimates for part (II) in the proof of \cite[Proposition 5.6]{BBM02}, there are $\hat C, \tilde C$ (depending on the Lipschitz constants of $\varphi$ and $\varphi'$), $\hat\gamma,\tilde \gamma>0$
and a random variable $n_4(\w)$ on a full measure subset of $\Omega$ (the same as in Corollary \ref{coroln4}) such that
 \begin{equation}\label{n4}
\left\{
\begin{array}{ll}
\Lambda_\w(\{k_\w>q(n)\})\leq \hat C e^{ -\hat \gamma q(n)}, & \forall n \textrm{ such that } q(n)\geq n_4(\w)\\
P(\{n_4(\w)>n\})\leq \tilde C e^{ -\tilde \gamma n},& \forall n\geq 1.
      \end{array}\right.
\end{equation}
 For the measure of the second part in \eqref{cdcd0} we have
  \begin{align}
  \Lambda_\w \left(
  \bigcup_{A \subset \{1,\ldots,q(n)\}}
  \bigcup_{\substack{n_j \geq K \\ \sum_A n_j \geq n/2}} Y_\w(A, n_j) \right)
  &
  \leq \sum_{A \subset \{1,\ldots,q(n)\}} \sum_{\substack{n_j \geq K\nonumber
  \\ \sum_A n_j \geq n/2}}  \prod_{j\in A} a_{n_j}
  \\&
  \leq  \sum_{0 \leq \ell \leq q(n)} \binom{q(n)}{\ell}
  \sum_{p=n/2}^\infty \left(b^{\star \ell} \right)_p\nonumber
  \\&
 \leq  2^{q(n)} \sum_{p=n/2}^\infty b_p.\label{stexp}
  \end{align}
Since $b_n=\mathcal O(e^{-\tilde\gamma_1 n^{\upsilon}})$, we have
 \begin{equation}\label{O}\sum_{p=n/2}^\infty b_p =\mathcal O(n^{1-\upsilon} e^{-\frac{\tilde\gamma}2 n^{\upsilon}}).\end{equation}
 We notice that the previous estimates are uniform over $\w$. The proof follows then by \eqref{cdcd0}, \eqref{n4} and \eqref{stexp} just by taking $n_2=n_4$ and considering $\alpha$ small enough.
\end{proof}

\subsubsection{Non-uniform decay of return times}\label
{nudrt}
We treat now the non-uniform decay of return times.

\begin{Lemma}\label{R to T}
 If there exist $C_i, \gamma_i$, $i=1,2$, and $0<\upsilon\leq1$, and for $P$ a.e. $\w$ a positive integer $g_0(\w)$, such that
\begin{equation}\label{return stret RtTnu}
\left\{
\begin{array}{ll}
        \leb(\{R_\w>n\})\leq C_1e^{-\gamma_1 n^\upsilon}, & \forall\,n\geq g_0(\w) \\
         P(\{g_0(\w)>n\})\leq C_2e^{-\gamma_3 n^\upsilon},& \forall\,n\geq 1,
      \end{array}\right.
\end{equation}
then there exist $C_j, \gamma_j$, $j=3,4$, and for $P$ a.e. $\w$ a positive integer $n_2(\w)$, such that
\begin{equation*}
\left\{
\begin{array}{ll}
       \Lambda_\w(\{T_\w>n\})\leq C_3e^{-\gamma_3 n^{\upsilon/2}}, & \forall\,n\geq n_2(\w) \\
         P(\{n_2(\w)>n\})\leq C_4e^{\gamma_4 n^{\upsilon/2}},& \forall\,n\geq 1,
      \end{array}\right.
\end{equation*}
\end{Lemma}
\begin{proof}
Let $0<\hat \upsilon<\upsilon$ to be fixed later. From \cite[Remark 3.1]{BBM02} one can see that estimates \eqref{return stret RtTnu} imply that there exist $C_3, C_3', \tilde C_3, \gamma_3, \gamma_3'>0$ and a random variable $n_3(\w)$ defined on a full $P$-measure subset of $\Omega$ so that
 \begin{equation}\label{n3}
\left\{
\begin{array}{ll}
        \leb(\{R_\w>n\})\leq C_3e^{-\gamma_3 n^{\upsilon}}\leb(\Delta_\w), & \forall\,n\geq n_3(\w) \\
        \leb(\Delta_\w)\leq n_3(\w) +\tilde C_3\\
         P(\{n_3(\w)>n\})\leq C_3'e^{-\gamma_3' n^\upsilon},& \forall\,n\geq 1.
      \end{array}\right.
\end{equation}
Then, from Lemma \ref{K2} we may consider  $a_n=\tilde C_1 e^{-\tilde\gamma_1 n^{\hat \upsilon}}$ to be so that $$\Lambda_\w(\{\tau_\w^j(x,x')-\tau_\w^{j-1}(x,x')>n\vert \tau_\w^{1},\ldots,\tau_\w^{j-1}\})\leq a_n n_3(\sigma^{\tau_\w^{j-1}+\ell_0}(\w)).$$
 Let $k_\w$, $Y_\w(A,n_j)$, $b_n$ and large $K$ be defined accordingly as in the proof of Lemma \ref{R to T unif} and set $\hat q(n)=\lfloor\alpha n^{\hat\upsilon}\rfloor$.
Conditioning
successively with respect to the different times, we get
  \begin{equation*}
  \Lambda_\w \bigl( Y_\w(A,n_j)\bigr) \leq \prod_{j\in A} \Lambda_\w(\{ \tau_\w^j-\tau_\w^{j-1}=n_j \vert
  \tau_\w^{1},\ldots,\tau_\w^{j-1}\})
  \leq \prod_{j\in A} a_{n_j} n_3(\sigma^{\tau_\w^{j-1}+\ell_0}(\w)).
  \end{equation*}
We recall that
  \begin{equation}
  \label{cdcd}
  \{ (x,x') : T_\w(x,x') > n\}
  \subset \{ k_\w(x,x')>\hat q(n)\} \cup
  \bigcup_{A \subset \{1,\ldots,\hat q(n)\}}
  \bigcup_{\substack{n_j \geq K \\ \sum_A n_j \geq n/2}} Y_\w(A, n_j),
  \end{equation}
and from \eqref{n4} there are $\hat C, \tilde C, \hat\gamma, \tilde\gamma>0$, $0<\varrho<1$ and a random variable $n_4(\w)$ such that
 \begin{equation}\label{n4 2}
\left\{
\begin{array}{ll}
\Lambda_\w(\{k_\w>\hat q(n)\})\leq  \hat C e^{ -\hat \gamma \hat q(n)}, & \forall n \textrm{ such that } \hat q(n)\geq n_4(\w)\\
P(\{n_4(\w)>n\})\leq \tilde C e^{ -\tilde \gamma n},& \forall n\geq 1.
      \end{array}\right.
\end{equation}
Let us now concentrate in the measure of the second part in \eqref{cdcd}. We have
  \begin{align}
  \Lambda_\w \left(
  \bigcup_{A \subset \{1,\ldots,\hat q(n)\}}
  \bigcup_{\substack{n_j \geq K \\ \sum_A n_j \geq n/2}} Y_\w(A, n_j) \right)
  &
  \leq \sum_{A \subset \{1,\ldots,\hat q(n)\}} \sum_{\substack{n_j \geq K\nonumber
  \\ \sum_A n_j \geq n/2}}  \prod_{j\in A} a_{n_j}n_3(\sigma^{\tau_\w^{j-1}+\ell_0}(\w))
  \\&
   \leq  \prod_{j=1}^{\hat q(n)} n_3(\sigma^{\tau_\w^{j-1}+\ell_0}(\w))\cdot 2^{\hat q(n)} \sum_{p=n/2}^\infty b_p.\label{stexp 2}
  \end{align}
 Fix $0<\rho<1$.  We say that $\Gamma\in\xi_\w^{\hat q(n)}$ is $n$-\emph{good} if for all $(x,x')\in\Gamma$ and $\ell\leq \hat q(n)$,
\begin{equation}\label{viol}\sum_{j=1}^\ell n_3(\sigma^{\tau_\w^{j-1}}(\w))^{\upsilon}\leq\rho n^{\upsilon}.\end{equation} The remaining cylinders are called $n$-\emph{bad}. We claim that there is a random variable $\hat n_3(\w)\geq n_3(\w)$ on a full measure subset of $\Omega$ such that, for all $\hat\upsilon'=\upsilon-\hat\upsilon$ and $n\geq \hat n_3(\w)$, the $\Lambda_\w$-measure of the $n$-\emph{bad} cylinders
is less than $\hat C_3e^{-\hat\gamma_3 n^{\hat\upsilon'}}$ and $P(\{\hat n_3(\w)>n\})\leq \hat C_3' e^{-\hat \gamma_3' n^{\hat\upsilon'}}$. Indeed, note first that from \eqref{n3}, for each fixed $1\leq\ell\leq \hat q(n)$ and $0=\tau^0,\tau^1,\ldots,\tau^{\ell-1}$,
\begin{equation}\label{aaa}\begin{array}{ll}\displaystyle
P\left(\left\{\sum_{j=1}^\ell n_3(\sigma^{\tau^{j-1}}(\w))^\upsilon>\rho n^\upsilon\right\}\right)&\displaystyle\leq \sum_{j=1}^\ell P\left(\left\{n_3(\sigma^{\tau_{j-1}}(\w))^\upsilon>\frac{\rho n^\upsilon}\ell\right\}\right)\\
&\displaystyle\leq C_3'\hat q(n) e^{-\gamma_3'\rho n^\upsilon/\hat q(n)}
\end{array}
\end{equation}
Let $M_n\subset \Omega\times\Delta$ be the set of points $(\w,x,x')$ such that $(x,x')$ belongs to an $n$-bad $\Gamma_\w\in\xi_\w^{\hat q(n)}$. Thus \eqref{aaa} implies $(P\times\Lambda)(M_n)\leq C_3'\hat q(n) e^{-\gamma_3'\rho n^\upsilon/\hat q(n)}$. Fix any $0<\eta<1$. Setting
$$M_n'=\left\{\w:\displaystyle\int_{\Delta_\w\times\Delta_\w}\chi_{M_n}\,d \Lambda_\w(x,x') > C_3'\hat q(n) e^{-\eta\gamma_3'\rho n^\upsilon/\hat q(n)} \right\},$$
we must have
\begin{equation}\label{bbb}P\left(M_n'\right)< e^{-(1-\eta)\gamma_3'\rho n^\upsilon/\hat q(n)},\end{equation}
otherwise, using Fubini's theorem we are lead into a contradiction.
Define, for each $n$,
$$B_n=\{\w: \exists\, k\geq n\, \textrm{ s.t. }\, \Lambda_\w(\{(x,x'):(\w,x,x')\in M_k\})>C_3'\hat q(k) e^{-\eta\gamma_3'\rho k^\upsilon/\hat q(k)} \}.$$
Then \eqref{bbb} implies that $P(B_n)\leq\sum_{k\geq n} e^{-(1-\eta)\gamma_3'\rho k^\upsilon/\hat q(k)} $, and therefore $\displaystyle\lim_{n\to\infty}P(B_n)=0$. For $P$ a.e. $\w\in\bigcup_n (\Omega\setminus B_n)$, we  set $\hat n_3(\w)=\sup\{n_3(\w),\inf\{n\geq 1: \w\notin B_n\}\}$.
We have then
\begin{align*}
P\left(\{\hat n_3(\w)>n \}\right) &\leq  P(\{\w:  \exists m>n \,\, \textrm{s.t.} \,\, \w\in M_m'\})          + P(\{n_3(\w)>n\})\\
&\leq \sum_{k>n}e^{-(1-\eta)\gamma_3'\rho k^\upsilon/\hat q(k)}+C_3'e^{-\gamma_3' n^\upsilon}\\
&\leq \hat C_3' e^{-\hat\gamma_3'n^{\hat\upsilon'}}.
\end{align*}
Moreover, the $\Lambda_\w$-measure of the $n$-bad cylinders is less than $C_3'\hat q(n) e^{-\eta\gamma_3'\rho n^\upsilon/\hat q(n)}$, for $n\geq \hat n_3(\w)$, proving the claim.
On the other hand, condition \eqref{viol} leads to
$$\prod_{j=1}^{\hat q(n)} n_3(\sigma^{\tau_\w^{j-1}+\ell_0}(\w))\leq e^{\hat q(n)\log(\rho n^{\hat\upsilon'})}.$$
Taking into account the claim and \eqref{O}, for $n\geq \hat n_3(\w)$ we have that \eqref{stexp 2} is less than $\tilde C e^{-\tilde\gamma n^{\hat\upsilon'}}$, if $\alpha$ is small enough, which
together with \eqref{cdcd} and \eqref{n4 2} implies that if $n\geq  n_2(\w)=\max\{\hat n_3(\w),n_4(\w)\}$,  we have $$\Lambda_\w(\{T_\w(x)>n\})\leq C_3e^{-\gamma_3 n^{\upsilon^*}},$$ with $\upsilon^*={\min\{\hat\upsilon,\hat\upsilon'\}}$. The optimal result occurs considering $\upsilon^*=\upsilon/2$.
\end{proof}

\subsection{From stopping times to the decay of correlations}\label{stdc}
\subsubsection{Matching} Due to the introduction of the \emph{waiting times} $n_i(\w)$ in the estimates over the joint stopping times at Lemmas \ref{R to T unif} and \ref{R to T}, we are not in conditions to capitalize the Gou\"{e}zel's \cite{G06} strategy to improve the estimates on the joint return times in \cite[Proposition 5.9]{BBM02}, that we recall in the following.

\begin{Lemma}\label{T to match}

Assume that there exist $C_i, \gamma_i$, $i=1,2$, and for $P$ a.e. $\w$ a positive integer $n_1(\w)$, such that
\begin{equation*}
\left\{
\begin{array}{ll}
       (\leb\times\leb)(\{T_\w>n\})\leq C_1e^{-\gamma_1 n^{\upsilon}}, & \forall\,n\geq n_1(\w) \\
         P(\{n_1(\w)>n\})\leq C_2e^{-\gamma_2 n^{\upsilon}},& \forall\,n\geq 1
      \end{array}\right..
\end{equation*}
Then there exist $C_3, \gamma_3 >0$ and a random variable $n_0(\w)$ on a full measure subset of $\Omega$, with $$P(\{n_0(\w)>n\})\leq C_3 e^{-\gamma_3n^{\upsilon/2}}$$ such that,
for each pair of absolutely continuous probability measures $\lambda, \lambda'$ in $\Delta$ with densities $\varphi,\varphi'\in\mathcal F_\beta^+\cap\mathcal L^\infty$,
there is $C$ (depending only on $\hat C_\varphi,\hat C_{\varphi'}, \tilde C_\varphi$ and $\tilde C_{\varphi'}$), and $\gamma>0$ such that for almost every $\w$ and $n\geq n_0(\w)$
$$\left|(F_\w^n)_*\lambda_\w -(F_\w^n)_*\lambda_\w'\right|\leq C e^{-\gamma n^{\upsilon/2}},$$
where $|\cdot|$ stands for total mass of a (signed) measure.
\end{Lemma}

\subsubsection{From equilibrium to decay of correlations}
Let us now point out how the rates on the convergence to equilibrium gives rise to similar estimates for the rates of decay of correlations. This finishes the proof of Theorem \ref{dacayrate}, which combined with Theorem \ref{random Markov structure} gives Theorems \ref{thA} and \ref{thB}. Let $\nu=\{\nu_\w\}$ be the family of measures given by Theorem \ref{exist.mu.induced}, and recall that $\rho_\w=d\nu_\w/d\leb$.
Consider $\varphi\in\mathcal L^\infty$. We assume first that $\psi\in\mathcal F_\beta^+\cap\mathcal L^\infty$, and let $b_\w=(\int_{\Delta_\w} \psi_\w \,d\nu_\w)^{-1}$, $\tilde{\psi}_\w={b_\w}\psi_\w $,  with $\int_{\Delta_\w} \tilde\psi_\w\,d\nu_\w =1$, and $\tilde\psi=\{\tilde\psi_\w\}$. Consider an absolutely continuous probability measure $\lambda$ on $\Delta$ such that $d\lambda_\w/d\leb=\tilde\psi_\w\rho_\w\in F_\beta^+\cap\mathcal L^\infty$. We have
\begin{eqnarray*}
\bar C^+_{\w}(\varphi,\psi,\nu,n)&=&\frac1{b_\w}\bar C^+_{\w}(\varphi,\tilde\psi,\nu,n)\\
&=&\frac1{b_\w}\left|\int_{\Delta_\w}(\varphi_{\sigma^n(\w)}\circ F_\w^n)\,d\lambda_\w -\int_{\Delta_{\sigma^n(\w)}}\varphi_{\sigma^n(\w)}\,d\nu_{\sigma^n(\w)}\int_{\Delta_\w}\tilde\psi_\w\,d\nu_\w \right|\\
&\leq&\frac{\tilde C_\varphi}{b_\w}\left|(F_\w^n)_*\lambda_\w -\nu_{\sigma^n(\w)}\right|\\
&\leq&\tilde C_\varphi \tilde C_\psi\left|(F_\w^n)_*\lambda_\w -\nu_{\sigma^n(\w)}\right|
\end{eqnarray*}
Finally, Lemmas \ref{R to T unif}, \ref{R to T} and Lemma \ref{T to match} lead to the desired estimates just by taking $\lambda'=\nu$, for some $C'=C'(\varphi,\psi)$. For $\psi\in\mathcal F_\beta$ we obtain the same estimates for $\hat\phi_\w=\psi_\w + C_\psi+1 \in F_\beta^+\cap\mathcal L^\infty$.

\bibliographystyle{novostyle}

\end{document}